\documentclass{article}

\usepackage[final]{graphicx} 
\usepackage{subcaption}    
\usepackage{color} 
\usepackage{varioref}  
\usepackage{rotating}  

\usepackage{pdflscape}		
\usepackage{setspace}		
\usepackage{mathrsfs}			
\usepackage{mathtools}
\usepackage{verbatim}			

\usepackage[final]{listings} 
\usepackage[pdftex,hidelinks]{hyperref}
\usepackage{cite} 

\usepackage{tikz,ulem} 
\usetikzlibrary{calc}
\usetikzlibrary{patterns}
\usetikzlibrary{positioning}
\usetikzlibrary{decorations.markings}
\usetikzlibrary{decorations.pathmorphing}
\usetikzlibrary{arrows}

\usepackage{latexsym}
\usepackage{amsfonts}
\usepackage{amssymb}
\usepackage{amsmath}
\usepackage{amsthm}
\usepackage{bbm}
\usepackage{afterpage}
\usepackage{fancyhdr}
\usepackage{lipsum} 
\usepackage{hyperref}
\usepackage{amsfonts}
\usepackage{amsthm}
\usepackage{amsmath}
\usepackage{amscd}
\usepackage{amssymb}
\usepackage{amsbsy}
\usepackage{epsf,graphicx,amsmath,verbatim,epsfig,amssymb,bbm}
\usepackage{tikz-cd,tikz, MnSymbol, relsize,enumitem, dsfont}
\usepackage{enumitem}
\usepackage{amsthm}

\usepackage[all]{xy}

\usepackage{pgfplots, tikz}

\usepackage{pgfplots}
\usepackage{tkz-euclide}
\usepgfplotslibrary{fillbetween}

\pgfplotsset{compat=1.6}

\pgfplotsset{soldot/.style={color=blue,only marks,mark=*}} \pgfplotsset{holdot/.style={color=blue,fill=white,only marks,mark=*}}

\newtheorem{theorem}{Theorem}[section]

\newtheorem{lemma}[theorem]{Lemma}
\newtheorem{corollary}[theorem]{Corollary}

\theoremstyle{definition}
\newtheorem{definition}[theorem]{Definition}

\newtheorem{convention}[theorem]{Convention}

\theoremstyle{remark}

\newcommand{\p}{\varphi}
\newcommand{\RR}{{\rm I\kern -1.6pt{\rm R}}}

\def\ds{\displaystyle}

\def\deq{:=}

\def\r{\widehat}
\def\R{\mathbb{R}}

\def\S{\mathbb{S}}

\def\id{\mathbbm{1}}

\def\H{\mathbb{H}}

\newcommand{\D}{\mathbb{D}}

\DeclareMathOperator{\Ima}{Im}

\def\H{\mathbb{H}}

\newcommand{\pilip}[1]{\pi_{#1}^{\text{Lip}}}

\newcommand{\length}[1]{\ell}

\newcommand{\ps}[2]{\mathcal{P}_{#1}}
\def\dps{d_{\mathcal{P}}}

\def\NN{\mathbb N}
\def\RR{\mathbb R}

\newcommand{\lip}[1]{\text{Lip}(#1)}

\def\lmin{\ell_{\text{min}}}

\newcommand{\core}[1]{{#1}_\infty}

\DeclareMathOperator{\Lip}{Lip}

\newcommand\blfootnote[1]{%
  \begingroup
  \renewcommand\thefootnote{}\footnote{#1}%
  \addtocounter{footnote}{-1}%
  \endgroup
}

\begin{document}

\title{The universal Lipschitz path space of the Heisenberg group $\H^1$}
\author{Daniel Perry}
{\let\newpage\relax\maketitle}

  \begin{abstract}

The goal of this paper is to define and inspect a metric version of the universal path space and study its application to purely 2-unrectifiable spaces, in particular the Heisenberg group $\H^1$. The construction of the universal Lipschitz path space, as the metric version is called, echoes the construction of the universal cover for path-connected, locally path-connected, and semilocally simply connected spaces. We prove that this universal Lipschitz path space of a purely 2-unrectifiable space, much like the universal cover, satisfies a unique lifting property, a universal property, and is Lipschitz simply connected. The existence of such a universal Lipschitz path space of $\H^1$ will be used to prove that $\pilip{1}(\H^1)$ is torsion-free in a subsequent paper. 

 \blfootnote{{\it Key words and phrases.} Heisenberg group, contact manifolds, unrectifiability, geometric measure theory, sub-Riemannian manifolds, universal path space, unique lifting property \\ {\bf Mathematical Reviews subject classification.} Primary: 53C17, 28A75 ; Secondary: 57K33, 54E35 \\ {\bf Acknowledgments.} This material is based upon work supported by the National Science Foundation under Grant Number DMS 1641020.  The author was also supported by NSF awards 1507704 and 1812055 and ARAF awards in 2022 and 2023.}
  \end{abstract}


\normalem

\section{Introduction}


In this paper, we define the universal Lipschitz path space of a based metric space and prove that for a based purely 2-unrectifiable space the universal Lipschitz path space, much like the universal cover, satisfies a unique lifting property and is simply connected, a result we state presently. A key application of the following main result is for the based metric space $(\H^1,0)$, which is the Heisenberg group based at the origin.

\begin{theorem}\label{main}
For a based purely 2-unrectifiable space $(M,p_0)$, there exists a based Lipschitz simply connected length space $(\ps{M}{d},[p_0])$ and a based Lipschitz map $\pi:(\ps{M}{d},[p_0])\rightarrow (M,p_0)$ that satisfies the unique lifting property. 
\end{theorem}

The space $\ps{M}{d}$ is called the \textbf{universal Lipschitz path space} of $M$. For the definition of the unique lifting property, see Definition~\ref{unique path lifting definition}. Indeed, we prove a more general version of Theorem~\ref{main}, where purely 2-unrectifiable space is replaced by tree-like homotopy space, which is defined in Definition~\ref{def TLH}. See Corollary~\ref{unique path lifting yields simply connected}. 

Moreover, for any metric space, the universal Lipschitz path space exists and is a pseudo-metric space, though it need not satisfy the unique lifting property. Provided that each point in the metric space supports only trivial local representation (see Definition~\ref{nonsingular def}), the universal Lipschitz path space is a metric space (Lemma~\ref{universal path space is a metric space}) and in fact a length space (Theorem~\ref{lem:length}).

Additionally, for a based metric space $(M,p_0)$, the definition of the pseudo-metric on the universal Lipschitz path space yields a pseudo-metric on the associated first Lipschitz homotopy group $\pilip{1}(M,p_0)$ (Definition~\ref{def ulps}), which is in fact a metric under a fairly mild assumption that the base point $p_0$ supports only trivial local representation. Thus, as the first Lipschitz homotopy group acts isometrically on $\ps{M}{d}$ (Theorem~\ref{free and isometric action}), under the mild assumption, $\pilip{1}(M,p_0)$ is a metric group. See Corollary~\ref{metric group}. This approach is relatively straightforward compared to efforts to topologize the standard fundamental group which can be quite involved; see for instance \cite{brazas2013fundamental}.

Covering space theory is an integral tool of algebraic topology. See any introductory text to algebraic topology, for instance \cite[Section 1.3]{hatcher2002algebraic}, for details. With its success in algebraic topology, the theory has been adjusted and imported into other fields with the aim of studying spaces that fail to satisfy the necessary semilocally simply connected property for the existence of a universal covering space (see `The Classification of Covering Spaces' in \cite{hatcher2002algebraic}). This paper represents such an effort, importing the ideas into metric topology. The universal Lipschitz path space is in fact a generalization of the universal covering space as they agree for locally Lipschitz simply connected spaces, such as Riemannian manifolds, simplicial complexes with a simplexwise linear metric, and quotients of Lipschitz simply connected Carnot groups.

Other efforts to adjust the universal covering space to instances where the base space is not assumed to be semilocally simply connected include \cite{berestovskiui2010covering} and \cite{Bog}. Berestovski\u{\i} and Plaut's work in \cite{berestovskiui2010covering} is of particular note as the aims of that paper are similar to the goals of establishing the universal Lipschitz path space. In \cite{berestovskiui2010covering}, the covering $\R$-tree of a metric space $M$ is defined in terms of ``non-backtracking'' rectifiable paths and is shown to be a metric tree that satisfies a unique lifting property and a universal property provided that $M$ is a length space. In this paper, the universal Lipschitz path space of a metric space $M$ is defined in terms of homotopy classes of Lipschitz paths and is shown to satisfy a unique lifting property and a universal property provided that $M$ is purely 2-unrectifiable. In the sequel to this paper, the universal Lipschitz path space is shown to be a metric tree given this assumption. Exploring the relationship between covering $\R$-trees and universal Lipschitz path spaces is the subject of ongoing research.

As the development and classification of covering spaces in algebraic topology is a tool for studying fundamental groups, the universal Lipschitz path space is a tool for studying the first Lipschitz homotopy group $\pilip{1}(M,p_0)$ of a based metric space $(M,p_0)$. Lipschitz homotopy groups are defined similarly to homotopy groups in algebraic topology with the additional requirement that all maps are Lipschitz rather than just continuous. See \cite[Definition~ 4.1]{Dej} for the original definition or \cite[Definition~28]{perry2020lipschitz} for an equivalent definition. 

We are primarily interested in $\pilip{1}(\H^1,0)$, the first Lipschitz homotopy group of the Heisenberg group $\H^1$ with the origin as the base point. Though the group $\pilip{1}(\H^1,0)$ is known to be uncountably generated \cite[Theorem~4.11(2)]{Dej}, there are interesting questions about the structure of the group which are still open. An example of such an open question was posed by Haj\l asz in \cite{lasz2018n+}: Is $\pilip{1}(\H^1,0)$ torsion-free? The tools developed in this paper will be used  to answer this question in the affirmative in the sequel to this paper. 

Alas, the analogy with covering space theory has limitations. Though we show there exists a universal Lipschitz path space of $\H^1$ that is Lipschitz simply connected and satisfies the unique lifting property, there is no Lipschitz simply connected covering space of $\H^1$, which is indicated presently. Indeed, as is the case in covering space theory, a metric space has to satisfy a metric notion of semilocally simply connectedness in order to have a Lipschitz simply connected covering space. 

\begin{definition}
A metric space $M$ is \emph{Lipschitz semilocally simply connected} if for any point $p\in M$, there exists an open neighborhood $U\subset M$ of the point $p$ such that the homomorphism induced by inclusion $U\hookrightarrow M$ between first Lipschitz homotopy groups based at $p$,
\[
\pilip{1}(U,p)\longrightarrow\pilip{1}(M,p),
\]
is the trivial map. 
\end{definition}

As will be shown, $\H^1$ is nowhere Lipschitz semilocally simply connected. Thus, there is no Lipschitz simply connected covering space of $\H^1$ and so no notion of a universal Lipschitz covering space of $\H^1$. 

Rather than focus solely on $\H^1$ here, we prove a more general result: contact 3-manifolds are nowhere Lipschitz semilocally simply connected. The Heisenberg group is the prototypical contact 3-manifold in that a contact 3-manifold $(M,\xi)$, when endowed with a sub-Riemannian metric, is a metric space that is locally modeled by the Heisenberg group, as is presently stated.
\begin{theorem}[BiLipschitz Theorem of Darboux, Corollary 23 in \cite{perry2020lipschitz}]\label{Darboux}
Let $(M,\xi)$ be a contact $3$-manifold endowed with a sub-Riemannian metric. For every $p\in M$, 
there exists a locally biLipschitz open distributional embedding $\p:\H^1\hookrightarrow M$ such that $\p(0)=p$. 
\end{theorem}
As such, since $\H^1$ is purely 2-unrectifiable \cite[Theorem~7.2]{Amb}, any contact 3-manifold is purely 2-unrectifiable as well \cite[Theorem~27]{perry2020lipschitz}. 
For a detailed account of contact 3-manifolds as sub-Riemannian manifolds, as well as some results involving their Lipschitz homotopy groups, see \cite{perry2020lipschitz}.

Before proving contact 3-manifolds are nowhere Lipschitz semilocally simply connected, we prove a more general version of Theorem 32 in \cite{perry2020lipschitz}.

\begin{lemma}\label{injective pi-1}
Let $(X,x_0)$ and $(Y,y_0)$ be based metric spaces where $Y$ is purely 2-unrectifiable. Let $\p:(X,x_0)\hookrightarrow(Y,y_0)$ be a based biLipschitz embedding. Then the homomorphism induced by $\p$ between the first Lipschitz homotopy groups
\[
\p_{\#}:\pilip{1}(X,x_0)\longrightarrow\pilip{1}(Y,y_0)
\]
is injective.
\end{lemma}

\begin{proof}

Let $\alpha:\S^1\rightarrow X$ be a Lipschitz map that represents an element of the kernel of the homomorphism $\p_\#$. So, there exists a Lipschitz map $H:\D^2\rightarrow Y$ such that $H$ restricted to the boundary is the Lipschitz map $\p\circ\alpha$:
\[
H|_{\partial\D^2}=\p\circ\alpha.
\]

As $Y$ is purely 2-unrectifiable, by Lemma 31 in \cite{perry2020lipschitz}, the Lipschitz map $H$ can be taken such that the image of $H$ is contained in the image of the Lipschitz map $\p\circ\alpha$. Thus, $H$ takes image entirely in the image of $\p$:
\[
\Ima(H)\subset\Ima(\p\circ\alpha)\subset\Ima(\p).
\]
Since the inverse $\varphi^{-1}:\Ima{\p}\rightarrow X$ is Lipschitz, the map given by composition
\[
\varphi^{-1}\circ H:\D^2\longrightarrow X
\]
is Lipschitz. The Lipschitz map $\varphi^{-1}\circ H$ when restricted to the boundary of $\D^2$ equals the map $\alpha$. Thus, $\alpha$ is Lipschitz null homotopic. Therefore, the only element in the kernel of $\varphi_\#$ is the trivial homotopy class.

\end{proof}

\begin{theorem}\label{no regular points}
Any contact 3-manifold $(M,\xi)$ endowed with a sub-Riemannian metric is nowhere Lipschitz semilocally simply connected.
\end{theorem}

\begin{proof}
Let $p\in M$ and take an open neighborhood $U\subset M$ of the point $p$. By \cite[Theorem~27]{perry2020lipschitz}, the contact 3-manifold $(M,\xi)$ is purely 2-unrectifiable.  As the inclusion map $(U,p)\hookrightarrow (M,p)$
is a based biLipschitz embedding, by Lemma~\ref{injective pi-1}, the homomorphism $\pilip{1}(U,p)\rightarrow\pilip{1}(M,p)$ induced by the inclusion between the first Lipschitz homotopy groups is injective. 

Via Theorem~\ref{Darboux}, there exists a based biLipschitz embedding $\p:(\H^1,0)\hookrightarrow (U,p)$ of the Heisenberg group into the neighborhood $U$ of $p$. So, by Lemma~\ref{injective pi-1}, the homomorphism $\p_\#:\pilip{1}(\H^1,0)\hookrightarrow\pilip{1}(U,p)$ is injective. Since the group $\pilip{1}(\H^1,0)$ is non-trivial \cite{Dej}, the group $\pilip{1}(U,p)$ is non-trivial as well. Thus, the homomorphism $\pilip{1}(U,p)\hookrightarrow\pilip{1}(M,p)$ is not the trivial map.

\end{proof}

The paper is organized as follows. In Section 2, we define the universal Lipschitz path space as a pseudo-metric space and discuss sufficient conditions for when it is a metric space. We also define the endpoint projection map associated with the universal Lipschitz path space and define a free and isometric action of the Lipschitz fundamental group of the original metric space on the universal Lipschitz path space before concluding that the first Lipschitz homotopy group is a metric group given mild assumptions. In Section 3, we discuss in general the unique lifting property, the unique path lifting property, Lipschitz simply connectedness, and the universal property for the universal Lipschitz path space. In Section 4, we define tree-like homotopy spaces, a generalization of pure 2-unrectifiability. We then prove that the unique path lifting property is satisfied for tree-like homotopy spaces before concluding that for such spaces, the unique lifting property is satisfied and the universal Lipschitz path space is Lipschitz simply connected.

\medskip
\noindent {\bf Acknowledgment}. The author wishes to thank Chris Gartland and Fedya Manin for their invaluable input on this paper. In addition to Gartland and Manin, the author would also like to thank David Ayala and Carl Olimb for their input, thoughts, and support as this paper came together.

\section{Definition of the universal Lipschitz path space}

\subsection{Background}

\begin{convention}
Throughout this paper, $I=[0,1]$ is the closed interval endowed with the Euclidean metric whose base point is $0\in I$. The product space $I\times I$ is endowed with the $L^1$ metric: for $(s,t),(s',t')\in I\times I$,
\[
d^1((s,t),(s',t'))=|s-s'|+|t-t'|,
\]
which is Lipschitz equivalent to the Euclidean metric on $I\times I$. For metric spaces $X$ and $Y$, the Lipschitz constant of a Lipschitz function $f:X\rightarrow Y$ is denoted by $\Lip(f)$. For a metric space $X$, the open ball centered at $x\in X$ of radius $r>0$ is denoted $B(x,r)$.
\end{convention}

Lipschitz paths will play a fundamental role in the definition and study of the universal Lipschitz path space. All paths and loops, unless stated otherwise, are Lipschitz and have domain $I$. For a metric space $M$ and a path $\gamma:[a,b]\rightarrow M$, the path $\gamma$ is said to \emph{join} initial point $\gamma(a)$ to the end point $\gamma(b)$. The \emph{length} of the path $\gamma$ is denoted $\length{M}(\gamma)$. 

We now present the definition of an arc length parametrized path. Note that for $\gamma:[a,b]\rightarrow M$, an arc length parametrized path, $\Lip(\gamma)=\frac{\length{M}(\gamma)}{b-a}$. 

\begin{definition}\label{arc length definition}
For a metric space $M$, a path $\gamma: [a,b]\rightarrow M$ is \emph{arc length parametrized} if for any $t_1,t_2\in[a,b]$,
\[
\length{M}\left(\left.\gamma\right|_{[t_1,t_2]}\right)=\frac{\length{M}(\gamma)}{b-a}~\left|t_1-t_2\right|.
\]
\end{definition}

For a path $\gamma$, the \emph{inverse path} is denoted $\overline{\gamma}$. For paths $\gamma_1$ and $\gamma_2$ such that the end point of $\gamma_1$ agrees with the initial point of $\gamma_2$, that is, $\gamma_1(1)=\gamma_2(0)$, the \emph{concatenation} is denoted $\gamma_1\ast\gamma_2$.

\begin{definition}\label{path classes}
For a metric space $M$, two paths $\gamma,\gamma':I\longrightarrow M$ are \emph{homotopic rel endpoints}, denoted $\gamma\simeq\gamma'$, if the initial points $\gamma(0)=\gamma'(0)$ and end points $\gamma(1)=\gamma'(1)$ of the paths agree and there exists a Lipschitz map $H:I\times I\rightarrow M$ such that
\[
H|_{I\times\{0\}}=\gamma,\hspace{.25cm} H|_{I\times\{1\}}=\gamma',\hspace{.25cm} H|_{\{0\}\times I}=\gamma(0),\hspace{.25cm}\text{and}\hspace{.25cm} H|_{\{1\}\times I}=\gamma(1).
\]
The map $H$ is a \emph{homotopy} from $\gamma$ to $\gamma'$. For a path $\gamma$, the class of all paths homotopic rel endpoints to $\gamma$ is denoted $[\gamma]_{\gamma(0)}^{\gamma(1)}$ and is referred to as the \emph{homotopy class} of $\gamma$. When it does not cause confusion, the endpoints will be dropped from the notation. For a homotopy class $[\gamma]$, a \emph{length minimizing path} $\core{\gamma}\in[\gamma]$ is a representative in the class such that
\[
\length{M}(\core{\gamma})=\inf\left\{~\length{M}(\gamma)~:~\gamma\in[\gamma]~\right\}.
\]
\end{definition}

%
%

The first Lipschitz homotopy group $\pilip{1}(M,p_0)$ of a based metric space $(M,p_0)$ is then the set of homotopy classes of loops $[\alpha]_{p_0}^{p_0}$ with group operation given by concatenation: $[\alpha_1]_{p_0}^{p_0}*[\alpha_2]_{p_0}^{p_0}\deq[\alpha_1*\alpha_2]_{p_0}^{p_0}$. This notion is equivalent to the definitions of first Lipschitz homotopy group which appear in \cite{Dej} and \cite{perry2020lipschitz}. Similar to fundamental groups, a Lipschitz map between based metric spaces induces a homomorphism between the associated first Lipschitz homotopy groups. 
\begin{definition}
For based metric spaces $(X,x_0)$ and $(Y,y_0)$ and a based Lipschitz  map $f:(X,x_0)\rightarrow(Y,y_0)$, the homomorphism induced by the map $f$ between first Lipschitz homotopy groups is
\begin{center}
$\begin{array}{rccl}
f_\#: & \pilip{1}(X,x_0) & \longrightarrow & \pilip{1}(Y,y_0) \\
     & [\alpha]  &   \longmapsto &   [f\circ\alpha]. 
\end{array}$
\end{center}
The map $f_\#$ is indeed well-defined and a homomorphism.
\end{definition}

Along with considering homotopy classes, we will often benefit from being able to relate the distance between points in a metric space with the lengths of paths between said points. We define a few types of metric spaces where such a relationship exists. Note that an immediate consequence of Definition~\ref{quasi convex} is that every quasi-convex space is Lipschitz path connected.

\begin{definition}\label{quasi convex}
A metric space $X$ with metric $d$ is \emph{quasi-convex} if there exists a real-value $C\geq 1$ such that for any pair of points $x_1,x_2\in X$, there exists a path $\gamma_X:I\rightarrow X$ where $\gamma_X(0)=x_1$ and $\gamma_X(1)=x_2$ such that
\[
\length{X}(\gamma_X)\leq C\cdot d(x_1,x_2).
\]
Such a value $C$ is called a \emph{quasi-convexity constant} of $X$. If $C$ can be taken to equal $1$, then $X$ is a \emph{geodesic space}. The metric space $X$ is a \emph{length space} if $C$ can be taken to be any value $C>1$, or equivalently, if for any $x_1,x_2\in X$,
\[
d(x_1,x_2)=\inf\left\{~\length{X}(\gamma)~:~\gamma:I\stackrel{\text{Lipschitz}}{\longrightarrow}X,~ \gamma(0)=x_1,~ \gamma(1)=x_2~\right\}.
\]
\end{definition}


\subsection{Metric structure on the universal path space}

For the remainder of this section, fix a based metric space $(M,p_0)$ with metric $d$. Additionally, since the first Lipschitz homotopy group and the universal Lipschitz path space only depend on the Lipschitz path component of the base point, assume that $M$ is Lipschitz path connected and has more than one element. We construct the universal Lipschitz path space $\ps{M}{d}$ for $(M,p_0)$, which is also a based space.

The definition of the universal Lipschitz path space as a set is similar to the set definition of the universal covering space from algebraic topology, in that elements of the universal Lipschitz path space are homotopy classes of paths, the meaningful distinction being that we are considering Lipschitz paths and homotopies rather than continuous paths and homotopies.

Moreover, we define a means of lifting the metric on the metric space to a pseudo-metric on the universal Lipschitz path space. The construction of the pseudo-metric, as well as several of the insights and results that follow, are inspired by the work of Bogley and Sieraski in \cite{Bog}.

\begin{definition}\label{def ulps}
The \emph{universal Lipschitz path space} $\ps{M}{d}$ of a based metric space $(M,p_0)$ as a set is
\[
\ps{M}{d}\deq\left\{~[\gamma]_{p_0}^{p}~:~\gamma:I\stackrel{\text{Lipschitz}}{\longrightarrow}M,~ \gamma(0)=p_0,~ \gamma(1)=p~\right\}.
\]
The base point of $\ps{M}{d}$ is $[p_0]$, the class of loops in $M$ based at $p_0$ that are Lipschitz null-homotopic. The \emph{lifted pseudo-metric} on the universal Lipschitz path space is defined to be, for $[\gamma_1],[\gamma_2]\in \ps{M}{d}$,
\[
\dps\left([\gamma_1],[\gamma_2]\right)\deq\inf\left\{~\length{M}(\beta)~:~\beta\simeq\overline{\gamma_1}*\gamma_2~\right\}.
\]
where the infimum is taken over all paths $\beta$ that are homotopic to $\overline{\gamma_1}*\gamma_2$ for some representatives $\gamma_1\in[\gamma_1]$ and $\gamma_2\in[\gamma_2]$.
\end{definition}

\begin{convention}
As we will only be discussing the universal Lipschitz path space with reference to metric spaces and Lipschitz maps, we will refer to the space $\ps{M}{d}$ simply as the universal path space.
\end{convention}

Note that for any based metric space $(M,p_0)$, the first Lipschitz homotopy group $\pilip{1}(M,p_0)\subset\ps{M}{d}$ is a subset of the universal path space $\ps{M}{d}$ as each element $[\alpha]\in\pilip{1}(M,p_0)$ is a homotopy class of paths whose endpoint agrees with its initial point $p_0$. Thus, Definition~\ref{def ulps} provides a pseudo-metric on the first Lipscthiz homotopy group when the lifted pseudo-metric $\dps$ is restricted to $\pilip{1}(M,p_0)\subset\ps{M}{d}$.



Now, the lifted pseudo-metric on the universal path space is not in general a metric. The harmonic archipelago is an example of a metric space whose lifted pseudo-metric on its universal path space is not a metric \cite[Example~1.1]{Bog}. For the harmonic archipelago, the lifted pseudo-metric fails to be a metric because of the existence of a point that supports a non-trivial homotopy class of loops that is locally representable (such points are called \emph{singular} in \cite{Bog}). 

\begin{definition}\label{nonsingular def}
For a metric space $M$ and a point $p\in M$, a homotopy class $[\alpha]\in\pilip{1}(M,p)$ is \emph{locally representable} if for every neighborhood $U\subset M$ of the point $p$, there exists a representative $\alpha_U\in[\alpha]$ such that the image of the representative $\Ima(\alpha_U)\subset U$ is a subset of $U$. If the only locally representable homotopy class for the point $p\in M$ is the trivial class $[p]$, then $p$ is said to \emph{support only trivial local representation}.
\end{definition}

As is indicated in the following lemma, the the lifted pseudo-metric when restricted to the first Lipschitz homotopy group is a metric when the base point supports only trivial local representation. Moreover, the lifted pseudo-metric is a metric on the universal path space as long as each point in the space supports only trivial local representation.


\begin{lemma}\label{universal path space is a metric space} 
Let $(M,p_0)$ be a based metric space. If the base point $p_0$ supports only trivial local representation, then the lifted pseudo-metric $\dps$ restricted to the first Lipschitz homotopy group $\pilip{1}(M,p_0)$ is a metric. Moreover, if each point in $M$ supports only trivial local representation, then the lifted pseudo-metric $\dps$ on the universal path space $\ps{M}{d}$ is a metric. In this case, $\dps$ is called the \emph{lifted metric}.
\end{lemma}


\begin{proof}


We will only show that the lifted pseudo-metric $\dps$ satisfies positive definiteness. Symmetry and the triangle inequality are obvious. In the following, the metric on $M$ is denoted by $d$.

First, suppose that $[\gamma_1],[\gamma_2]\in \ps{M}{d}$  are equal homotopy classes, $[\gamma_1]=[\gamma_2]$. So,  representatives $\gamma_1\in[\gamma_1]$ and $\gamma_2\in[\gamma_2]$ are homotopic rel endpoints. Thus, the paths $\overline{\gamma_1}$ and $\gamma_2$ concatenate to a loop $\overline{\gamma_1}*\gamma_2$ based at $\gamma_1(1)=\gamma_2(1)$ that is null-homotopic. Therefore, $\dps([\gamma_1],[\gamma_2])=0$. 

Now, suppose for homotopy classes $[\gamma_1],[\gamma_2]\in \ps{M}{d}$ that 
$\dps([\gamma_1],[\gamma_2])=0$.
%
%
%
Let $\varepsilon>0$ be given. Since the lifted pseudo-metric is defined as an infimum, there exists a path $\alpha_\varepsilon$ with initial point $\alpha_\varepsilon(0)=\gamma_1(1)$ and end point $\alpha_\varepsilon(1)=\gamma_2(1)$
such that $\alpha_\varepsilon\simeq\overline{\gamma_1}*\gamma_2$ and $\length{M}(\alpha_\varepsilon)<\varepsilon$. For any $t\in I$,
\begin{eqnarray*}
d(\gamma_1(1),\alpha_\varepsilon(t)) 
 &\leq& \length{M}\left(\alpha_\varepsilon|_{[0,t]}\right) \\ &\leq& \length{M}(\alpha_\varepsilon) \\ &<& \varepsilon.
\end{eqnarray*}
Thus, the image $\Ima(\alpha_\varepsilon)$ is a subset of the open ball $B(\gamma_1(1),\varepsilon)$ centered at $\gamma_1(1)$ of radius $\varepsilon$.
%

If $[\gamma_1],[\gamma_2]\in\pilip{1}(M,p_0)$, then it is immediate that $\left[\overline{\gamma_1}*\gamma_2\right]\in\pilip{1}(M,p_0)$ and the path $\alpha_\varepsilon$ is a loop. Otherwise, since $\varepsilon>0$ is an arbitrary positive value,
the initial point $\alpha_\varepsilon(0)=\gamma_1(1)$ is arbitrarily close to the end point $\alpha_\varepsilon(1)=\gamma_2(1)$, that is, $d(\gamma_1(1),\gamma_2(1))<\varepsilon$ for all $\varepsilon>0$. Thus, the end points of $\gamma_1$ and $\gamma_2$ are equal, $\gamma_1(1)=\gamma_2(1)$. So, $\left[\overline{\gamma_1}*\gamma_2\right]$ is a class of loops centered at $\gamma_1(1)$ and, for any $\varepsilon>0$, the path $\alpha_\varepsilon$ is a loop. 

Now, for every open ball $B(\gamma_1(1),\varepsilon)$ centered at $\gamma_1(1)$, the class of loops $[\overline{\gamma_1}*\gamma_2]$ has a representative $\alpha_\varepsilon$ whose image is a subset of the open ball. Since open balls form a basis for the topology on $M$, the class $[\overline{\gamma_1}*\gamma_2]$ is locally representable. So, whether only the base point $p_0$ supports only trivial local representation and $[\gamma_1],[\gamma_2]\in\pilip{1}(M,p_0)$ or all points support only trivial local representation, $[\overline{\gamma_1}*\gamma_2]$ is the trivial class based at $\gamma_1(1)$. Therefore, the paths $\gamma_1$ and $\gamma_2$ are homotopic rel endpoints, that is, $[\gamma_1]=[\gamma_2]$.
\end{proof}

\subsection{Endpoint projection and an action of $\pilip{1}(M,p_0)$}



Along with the pseudo-metric, there is a based map from the based universal path space $(\ps{M}{d},[p_0])$ to the based metric space $(M,p_0)$ given by endpoint projection:
\begin{center}
$\begin{array}{rccl}
\pi: & \ps{M}{d} & \rightarrow & M \\
     & [\gamma]  &   \mapsto &   \gamma(1). 
\end{array}$
\end{center}
The map $\pi$ is well-defined as all paths contained in the homotopy class $[\gamma]$ have the same endpoint. Additionally, the map $\pi$ is based as $\pi([p_0])=p_0$. As is now shown, endpoint projection is a Lipschitz map.


\begin{lemma}\label{endpoint projection is a metric map}
Endpoint projection $\pi:\ps{M}{d}  \rightarrow  M$ is Lipschitz with 
\[
0<\Lip(\pi)\leq 1.
\]
Additionally, if $M$ is a geodesic space, then $\Lip(\pi)=1$.
\end{lemma} 

\begin{proof}
Let $[\gamma_1]_{p_0}^{p_1},[\gamma_2]_{p_0}^{p_2}\in \ps{M}{d}$. Note that if a path $\beta$ is homotopic rel endpoints to the concatenation $\overline{\gamma_1}*\gamma_2$, then $\beta$ is a path with initial point $p_1$ and endpoint $p_2$.  Then, since the distance $d(p_1,p_2)$ between points $p_1$ and $p_2$ is bounded above by the length $\length{M}(\beta)$ of any path $\beta$ joining them, 
\begin{eqnarray*}
d(\pi[\gamma_1]_{p_0}^{p_1},\pi[\gamma_2]_{p_0}^{p_2}) & = & d(p_1,p_2) \\
 &\leq & \inf\{\length{M}(\beta)~:~\beta\text{ joins }p_1\text{ and }p_2\} \\
 & \leq & \inf\{\length{M}(\beta)~:~\beta\simeq\overline{\gamma_1}*\gamma_2\text{ rel endpoints }p_1\text{ and }p_2\}\\
 & = & \dps\left([\gamma_1]_{p_0}^{p_1},[\gamma_2]_{p_0}^{p_2}\right).
\end{eqnarray*}
Therefore, the map $\pi$ is Lipschitz and $\Lip(\pi)\leq 1$.

Now, by assumption, the metric space $M$ is Lipschitz path connected and has more than one element. Thus, endpoint projection $\pi$ is not a constant map. So, $\Lip(\pi)>0$.

Additionally, assume that $M$ is a geodesic space. Let $p_1, p_2\in M$ be distinct points in $M$. Let $\gamma_M:I\rightarrow M$ be a geodesic joining $p_1$ to $p_2$. Now, let $\gamma_1$ be a path in $M$ joining $p_0$ to $p_1$. So, $\gamma_2=\gamma_1*\gamma_M$ is a path in $M$ joining $p_0$ to $p_2$. Then, considering $[\gamma_1]_{p_0}^{p_1}$ and $ [\gamma_2]_{p_0}^{p_2}$, since $\gamma_M\simeq\overline{\gamma_1}*\gamma_2$,
\begin{eqnarray*}
\dps([\gamma_1]_{p_0}^{p_1},[\gamma_2]_{p_0}^{p_2} )& = & \inf\{~\length{M}(\beta)~:~\beta\simeq\overline{\gamma_1}*\gamma_2~\} \\
& = & \length{M}(\gamma_M) \\
& = & d(p_1,p_2) \\
& = & d(\pi[\gamma_1]_{p_0}^{p_1},\pi[\gamma_2]_{p_0}^{p_2})\end{eqnarray*}
Thus, $\Lip(\pi)=1$. 

\end{proof}

There is also a left group action of the first Lipschitz homotopy group $\pilip{1}(M,p_0)$ on the universal path space $\ps{M}{d}$ given by concatenation:
\begin{center}
$\begin{array}{rccl}
\ast: & \pilip{1}(M,p_0)\times\ps{M}{d} & \longrightarrow & \ps{M}{d} \\ \\
 & ([\alpha]_{p_0}^{p_0}, [\gamma]_{p_0}^p) & \longmapsto &[\alpha\ast\gamma]_{p_0}^p .
\end{array}$
\end{center}

The stated map is indeed a left group action as, for any $[\gamma]\in\ps{M}{d}$,
\[
[p_0]\ast[\gamma] = [p_0\ast\gamma] = [\gamma],
\]
and, for any $[\alpha_1],[\alpha_2]\in\pilip{1}(M,p_0)$ and $[\gamma]\in\ps{M}{d}$,
\begin{eqnarray*}
[\alpha_1]\ast([\alpha_2]\ast[\gamma]) & = & [\alpha_1]\ast([\alpha_2\ast\gamma]) \\
							    & = & [\alpha_1\ast(\alpha_2\ast\gamma)] \\
							    & = & [(\alpha_1\ast\alpha_2)\ast\gamma] \\
							    & = & [\alpha_1\ast\alpha_2]\ast[\gamma] \\
							    & = & ([\alpha_1]\ast[\alpha_2])\ast[\gamma].
\end{eqnarray*}

Additionally, the group action preserves the fibers of the endpoint projection $\pi$ as, for $[\alpha]_{p_0}^{p_0}\in\pilip{1}(M,p_0)$ and $[\gamma]_{p_0}^p\in\ps{M}{d}$,
\[
\pi\left([\alpha]_{p_0}^{p_0}\ast[\gamma]_{p_0}^p \right) = \pi\left([\alpha\ast\gamma]_{p_0}^p \right) = p =\pi\left([\gamma]_{p_0}^p\right).
\]

\begin{theorem}\label{free and isometric action}
The first Lipschitz homotopy group $\pilip{1}(M,p_0)$ acts freely and isometrically on the universal path space $\ps{M}{d}$.
\end{theorem}

\begin{proof}
First, to show that the group action is free, let $[\gamma]\in\ps{M}{d}$ and $[\alpha]\in\pilip{1}(M,p_0)$ where $[\alpha]\ast[\gamma]=[\gamma]$. By the definition of the group action, the paths $\alpha\ast\gamma\simeq\gamma$ are homotopic. Recall that the path $\overline{\gamma}$ is the reverse of the path $\gamma$ and that $\gamma\ast\overline{\gamma}$ is a null homotopic loop based at $p_0$. The loop $\alpha$ is thus null homotopic as
\begin{eqnarray*}
\alpha    & \simeq & \alpha\ast p_0 \\
		& \simeq & \alpha\ast (\gamma\ast\overline{\gamma}) \\
		& \simeq & (\alpha\ast \gamma)\ast\overline{\gamma} \\
		& \simeq & \gamma\ast\overline{\gamma} \\
		& \simeq & p_0.
\end{eqnarray*}
Thus, since $[\alpha]=[p_0]$, the group action is free.

Now, to show that the group action is isometric, let $[\gamma_1],[\gamma_2]\in\ps{M}{d}$ and $[\alpha]\in\pilip{1}(M,p_0)$. Note the following equivalences, where $\beta$ is a path in $M$:
\begin{center}
$\begin{array}{lcl}
\beta\simeq\overline{\gamma_1}\ast\gamma_2 & \Leftrightarrow  & \beta\simeq\overline{\gamma_1}\ast p_0\ast\gamma_2 \\
& \Leftrightarrow  & \beta\simeq\overline{\gamma_1}\ast (\overline{\alpha}\ast\alpha)\ast\gamma_2 \\
& \Leftrightarrow  & \beta\simeq(\overline{\gamma_1}\ast \overline{\alpha})\ast(\alpha\ast\gamma_2) \\
& \Leftrightarrow  & \beta\simeq(\overline{\alpha\ast\gamma_1})\ast(\alpha\ast\gamma_2). \\
\end{array}$
\end{center}
The final step follows from the equality of paths $\overline{\gamma_1}\ast \overline{\alpha}=\overline{\alpha\ast\gamma_1}$.

Therefore, the infimums appearing from the definition of the lifted pseudo-metric $\dps$ in
\[
\dps([\gamma_1],[\gamma_2]) = \inf\left\{\length{M}(\beta)~:~\beta\simeq\overline{\gamma_1}\ast\gamma_2 \right\}
\]
and
\[
\dps([\alpha\ast\gamma_1],[\alpha\ast\gamma_2]) = \inf\left\{\length{M}(\beta)~:~\beta\simeq(\overline{\alpha\ast\gamma_1})\ast(\alpha\ast\gamma_2) \right\}
\]
are determined by the same set of paths. As such, the distances
\[
\dps([\gamma_1],[\gamma_2]) = \dps([\alpha\ast\gamma_1],[\alpha\ast\gamma_2])
\]
are equal and thus the group action is isometric.
\end{proof}

Following from Theorem~\ref{free and isometric action}, the first Lipschitz homotopy group $\pilip{1}(M,p_0)$ acts isometrically on itself when the action is restricted to $\pilip{1}(M,p_0)\subset\ps{M}{d}$. Thus, by Lemma~\ref{universal path space is a metric space} and Theorem~\ref{free and isometric action},

\begin{corollary}\label{metric group}
If $p_0$ supports only trivial local representation, then the first Lipschitz homotopy group $\pilip{1}(M,p_0)$ is a metric group with respect to the lifted metric $\dps$. 
\end{corollary}

\section{Unique lifting property and the universal path space}

We now discuss lifts of Lipschitz maps with respect to endpoint projection. In this section, as in the last, fix a based metric space $(M,p_0)$ with metric $d$. Additionally, assume that $M$ is Lipschitz path connected and contains more than one element. 

\begin{definition}
Let $(X,x_0)$ be a based metric space and $f:(X,x_0)\rightarrow (M,p_0)$ be a based Lipschitz map. A \emph{lift} of the map $f$ is a based Lipschitz map $F:(X,x_0)\rightarrow(\ps{M}{d},[p_0])$ such that the following diagram commutes:
\begin{center}
\begin{tikzcd}
 && \ps{M}{d} \arrow[dd, "\pi"] \\ \\
 X \arrow[rr, "f"']\arrow[uurr, dashed, "F"] && M.
\end{tikzcd}
\end{center}
\end{definition}

The following result comparing Lipschitz constants of the Lipschitz map $f$ and the lift $F$ is immediate, following from Lemma~\ref{endpoint projection is a metric map}.

\begin{lemma}\label{lift lip bound}
Let $X$ be a based metric space with metric $d_X$ and $f:X\rightarrow M$ be a based Lipschitz map. Let $F:X\rightarrow\ps{M}{d}$ be a lift of the map $f$. Then
\[
\Lip(f)\leq\Lip(F).
\]
\end{lemma}

\begin{proof}
Let $x_1,x_2\in X$. Then,
\begin{eqnarray*}
d(f(x_1),f(x_2)) & = & d(\pi\circ F(x_1),\pi\circ F(x_2)) \\
						& \leq & \dps( F(x_1), F(x_2)) \\
						& \leq & \Lip(F)\cdot d_X(x_1,x_2).
\end{eqnarray*}
Thus, the Lipschitz constant $\Lip(f)$ is no larger than $\Lip(F)$. 

\end{proof}

\subsection{Lifts of paths to the universal path space}\label{def of lift}


Now, we show that there exist lifts of paths in $M$ to the universal path space. Let $\gamma:(I,0)\rightarrow(M,p_0)$ be a based path. For any $t\in I$, there exists a reparametrization $\gamma_t$ of the path $\gamma$ such that the endpoint of the reparametrization is the point $\gamma(t)$:
\begin{center}
$\begin{array}{rccl}
\gamma_t: & (I,0) & \rightarrow & (M,p_0) \\
 & s & \mapsto & \gamma(ts).
\end{array}$
\end{center}
The path $\gamma_t$ then determines an element of the universal path space, $[\gamma_t]_{p_0}^{\gamma(t)}\in\ps{M}{d}$. Thus, there is a map into the universal path space given by:
\begin{center}
$\begin{array}{rccl}
\r{\gamma}: & (I,0) & \rightarrow & (\ps{M}{d},[p_0]) \\
& t & \mapsto & \ds[\gamma_t]_{p_0}^{\gamma(t)}.
\end{array}$
\end{center}
The map is based as $\r{\gamma}(0)=[p_0]$. Also, 
for any $t\in I$,
\[
\pi\circ\r{\gamma}(t)=\pi\left([\gamma_t]_{p_0}^{\gamma(t)}\right)=\gamma(t).
\]
The next lemma shows that the map $\r{\gamma}$ is indeed Lipschitz and thus a lift of the path $\gamma$.

\begin{lemma}\label{the lift is Lipschitz}
Let $\gamma:(I,0)\longrightarrow(M,p_0)$ be a based path. Then, the map $\r{\gamma}$
is Lipschitz.
 Thus, $\r{\gamma}$ is a lift of the path $\gamma$. Additionally, $\Lip(\r{\gamma})=\Lip(\gamma)$.
\end{lemma}

\begin{proof}
Let $t_1,t_2\in I$. The restricted path $\gamma|_{[t_1,t_2]}$ is homotopic rel endpoints to the concatenation $\overline{\gamma_{t_1}}*\gamma_{t_2}$. So,
\begin{eqnarray*}
\dps(\r{\gamma}(t_1),\r{\gamma}(t_2)) &=& \inf\left\{~\length{M}(\beta)~:~\beta\simeq\overline{\gamma_{t_1}}*\gamma_{t_2}~\right\} \\
									&\leq  & \length{M}(\gamma|_{[t_1,t_2]}) \\
									&\leq & \Lip(\gamma)\cdot |t_1-t_2|.
\end{eqnarray*}
Thus, the map $\r{\gamma}$ is indeed Lipschitz with $\Lip(\r{\gamma})\leq\Lip(\gamma)$. So, $\r{\gamma}$ is a lift of $\gamma$. Additionally, by Lemma~\ref{lift lip bound}, $\Lip(\r{\gamma})\geq\Lip(\gamma)$. Therefore, $\Lip(\r{\gamma})=\Lip(\gamma)$.
\end{proof}

\subsection{Unique path lifiting and unique lifting properties}

Having determined that lifts of paths exist with reference to the universal path space, we show in Theorem~\ref{unique lifting} what conditions are sufficient for a general Lipschitz map to have a lift. 

As we proceed, we will also be interested in knowing if the lift of any given map is unique. When all lifts of paths are unique, the universal path space is said to have the unique path lifting property, which is defined now. 

\begin{definition}
Endpoint projection $\pi:(\ps{M}{d},[p_0])\rightarrow(M,p_0)$ has the \emph{unique path lifting property} if for any based path $\gamma:(I,0)\rightarrow(M,p_0)$, there exists a unique lift $\r{\gamma}:(I,0)\rightarrow(\ps{M}{d},[p_0])$ such that the following diagram commutes:
\begin{center}
\begin{tikzcd}
 & & \ps{M}{d} \arrow[dd, "\pi"] \\ \\
 I \arrow[rr, "\gamma"'] \arrow[rruu, dashed, "{\r{\gamma}}", "\exists!"'] & & M.
\end{tikzcd}
\end{center}
When this is the case, the universal path space $\ps{M}{d}$ is said to satisfy the unique path lifting property.
\end{definition}

In the last section of this paper, we show that the universal path space of any $\D^2$-unrectifiable metric space, in particular the Heisenberg group $\H^1$, satisfies the unique path lifting property. See Theorem~\ref{unique path lifting}.

In addition to describing sufficient conditions for lifting a Lipschitz map, we also show that the lifts guaranteed by Theorem~\ref{unique lifting} are unique whenever the universal path space satisfies the unique path lifting property.


{



%
%

\begin{theorem}\label{unique lifting}
Let $(X,x_0)$ be a based quasi-convex metric space with metric $d_X$. Let $f:(X,x_0)\rightarrow(M,p_0)$ be a based Lipschitz map such that the homomorphism induced by $f$ between first Lipschitz homotopy groups
\[
f_\#:\pilip{1}(X,x_0) \longrightarrow \pilip{1}(M,p_0)
\]
is trivial at $[p_0]$. Then, there exists a lift $\r{f}$ of the map $f$:
\begin{center}
\begin{tikzcd}
 & & \ps{M}{d} \arrow[dd, "\pi"] \\ \\
 X \arrow[rr, "f"] \arrow[uurr, dashed, "{\r{f}}"] & & M. 
\end{tikzcd}
\end{center}
Additionally, if $X$ is a length space, then $\lip{\r{f}}=\lip{f}$. Moreover, if the universal path space $\ps{M}{d}$ satisfies the unique path lifting property, then the lift $\r{f}$ is unique.
\end{theorem}

\begin{proof}

We will begin by defining a lift of the map $f$ and arguing that the lift is well-defined. The argument will be the standard such argument. Then, we will show that the lift is indeed Lipschitz which follows from the map $f$ being Lipschitz and the metric space $X$ being quasi-convex. Finally, with the assumption that $\pi:(\ps{M}{d},[p_0])\rightarrow(M,p_0)$ satisfies the unique path lifting property, we show that the lift is unique.

We now define a map $\r{f}:X\rightarrow\ps{M}{d}$ using that the metric space $X$ is Lipschitz path-connected. Let $x\in X$. Since $X$ is Lipschitz path-connected, there exists a based path $\gamma:(I,0)\rightarrow (X,x_0)$ joining the base point $x_0$ to $x=\gamma(1)$. Thus, the composition $f\circ\gamma$ is a path in $M$ joining the base point $p_0$ to the point $f(x)$. So, by the construction in Section \ref{def of lift}, there is a lift $\r{f\circ\gamma}$ of the path $f\circ\gamma$. Define the map $\r{f}$ as sending the point $x\in X$ to the endpoint of this lifted path $\r{f\circ\gamma}$:
\begin{center}
$\begin{array}{rccc}
\r{f}: & (X,x_0) & \rightarrow & (\ps{M}{d},[p_0]) \\
 & x & \mapsto & \r{f\circ\gamma}(1).
\end{array}$
\end{center}

To show that the map $\r{f}$ is well-defined, let $\gamma':(I,0)\rightarrow (X, x_0)$ be another based path joining $x_0$ and $x$. Then, the concatenation $\gamma'*\overline{\gamma}$ is a loop in $X$ that is based at $x_0$ and post-composing by $f$ yields a loop in $M$ based at $p_0$:
\[
f\circ(\gamma'*\overline{\gamma})=(f\circ\gamma')*(\overline{f\circ\gamma}).
\]
Since $f$ induces the trivial homomorphism between first Lipschitz homotopy groups, the loop $\ds(f\circ\gamma')*(\overline{f\circ\gamma})$ is null-homotopic. Thus, the paths $f\circ\gamma'$ and $f\circ\gamma$ joining $p_0$ to $f(x)$ are homotopic rel endpoints. Therefore, the map $\r{f}$ is well-defined as we have the following equality in $\ps{M}{d}$:
\[
\r{f\circ\gamma'}(1)=[f\circ\gamma']=[f\circ\gamma]=\r{f\circ\gamma}(1).
\]

An immediate consequence of $\r{f}$ being well-defined is that the map is based. Indeed, when determining $\r{f}(x_0)$, the path $\gamma$ can be taken to be the constant path at $x_0$. Then, the path $f\circ\gamma$ is the constant path at $p_0$. So, the lifted path $\r{f\circ\gamma}$ is the constant path at $[p_0]$ and thus, $\r{f}(x_0)=\r{f\circ\gamma}(1)=[p_0]$. 

Before showing that the map $\r{f}$ is Lipschitz, note that since $\r{f\circ\gamma}$ is a lift of the path $f\circ\gamma$ in $M$,
\[
\pi\circ \r{f}(x) = \pi\circ \r{f\circ\gamma}(1)=f\circ\gamma(1)=f(x).
\]
Thus, provided that the map is Lipschitz, $\r{f}$ is indeed a lift of the given map $f$.


Let $x_1,x_2\in X$. Since the metric space $X$ is quasi-convex, by Definition~\ref{quasi convex}, there exists a path $\gamma_X$ joining $x_1$ to $x_2$ such that $\length{X}(\gamma_X)\leq C\cdot d_X(x_1,x_2)$, where $C$ is the quasi-convexity constant for $X$. 

Consider the distance $\dps(\r{f}(x_1),\r{f}(x_2))$. Since the classes $\r{f}(x_1)$ and $\r{f}(x_2)$ do not depend on the paths joining the base point $x_0$ to $x_1$ and $x_2$ respectively, neither does the distance. With this in mind, let $\gamma_1$ be a path in $X$ joining $x_0$ to $x_1$, which exists since $X$ is Lipschitz path-connected. Also, let $\gamma_2=\gamma_1*\gamma_X$, which is a path in $X$ joining $x_0$ to $x_2$. Then, we can rewrite the distance as 
\begin{equation*}
\dps(\r{f}(x_1),\r{f}(x_2))  =  \dps([f\circ\gamma_1],[f\circ\gamma_2]).
\end{equation*}

Note, by design $\overline{\gamma_1}*\gamma_2\simeq\gamma_X$. So, we have
\[
(\overline{f\circ\gamma_1})*(f\circ\gamma_2)=f\circ(\overline{\gamma_1}*\gamma_2)\simeq f\circ\gamma_X.
\]
This equivalence, along with the map $f$ being Lipschitz, yields the following string of inequalities:
\begin{eqnarray*}
\dps(\r{f}(x_1),\r{f}(x_2)) & = &  \dps([f\circ\gamma_1],[f\circ\gamma_2]) \nonumber \\
& = & \inf_{\alpha}\{~\length{M}(\alpha)~:~\alpha\simeq(\overline{f\circ\gamma_1})*(f\circ\gamma_2)~\} \nonumber \\
    & \leq & \length{M}(f\circ\gamma_X)  \nonumber \\
   & \leq & \lip{f}\cdot\length{X}(\gamma_X) \nonumber \\
   & \leq & \lip{f}\cdot C\cdot d_X(x_1,x_2) 
\end{eqnarray*}

Therefore, $\r{f}$ is Lipschitz with  $\lip{\r{f}}\leq C\cdot\lip{f}$. Moreover, if the metric space $X$ is a length space, then $\lip{\r{f}}\leq \lip{f}$. Additionally, by Lemma~\ref{lift lip bound}, $\lip{f}\leq\lip{\r{f}}$. Thus the Lipschitz constants for the maps $\r{f}$ and $f$ are equal.

Finally, assume that the universal path space $\ps{M}{d}$ satisfies the unique path lifting property. Let $F:X\rightarrow\ps{M}{d}$ be a lift of $f$ and take a point $x\in X$. Since $X$ is Lipschitz path-connected, there exists a  path $\gamma$ joining the base point $x_0$ to $x$. The composition $F\circ\gamma:I\rightarrow \ps{M}{d}$ is then a lift of the path $f\circ\gamma$. By assumption, lifts of paths are unique. Thus, we have an equality of paths $F\circ\gamma=\r{f\circ\gamma}$. Therefore, the lift $F$ is equal to the lift $\r{f}$ as
\[
F(x)=F\circ\gamma(1)=\r{f\circ\gamma}(1)=\r{f}(x)
\]
for any $x\in X$.
\end{proof}

\subsection{Unique lifting property and Lipschitz simply connected universal path space}


With the sufficient conditions for the existence of a lift from Theorem~\ref{unique lifting} in mind, We define what it means for a based Lipschitz map to satisfy the unique lifting property.

\begin{definition}\label{unique path lifting definition}
Let $(E,e_0)$ be a based metric space. A based Lipschitz map $r:(E,e_0)\rightarrow (M,p_0)$ satisfies the \emph{unique lifting property} if, for any based, quasi-convex metric space $(X,x_0)$ and any based Lipschitz map $f:(X,x_0)\rightarrow(M,p_0)$ such that the homomorphism induced on first Lipschitz homotopy groups $f_\#:\pilip{1}(X,x_0)\rightarrow\pilip{1}(M,p_0)$ is trivial, there exists a unique based Lipschitz map $\r{f}:(X,x_0)\rightarrow(E,e_0)$ such that the following diagram commutes:
\begin{center}
\begin{tikzcd}
 & & E \arrow[dd, "r"] \\ \\
 X \arrow[rr, "f"] \arrow[uurr, dashed, "\exists!"', "{\r{f}}"] & & M. 
\end{tikzcd}
\end{center}
When such a map $\r{f}$ exists, the map is called a \emph{lift} with respect to the map $r:(E,e_0)\rightarrow (M,p_0)$.
\end{definition}

Following from definitions, if the map $\pi:\ps{M}{d}\rightarrow M$ satisfies the unique lifting property, then $\pi$ satisfies the unique path lifting property. By Theorem~\ref{unique lifting}, the converse is true as well.

As we now show, if $\pi:(\ps{M}{d},[p_0])\rightarrow (M,p_0)$ satisfies the unique lifting property, then the universal path space $\ps{M}{d}$ is Lipschitz simply connected.

\begin{lemma}\label{induced r is injective}
Let $(E,e_0)$ be a based metric space. If a based Lipschitz map $r:(E,e_0)\rightarrow (M,p_0)$ satisfies the unique lifting property, the homomorphism induced by $r$ between first Lipschitz homotopy groups,
\[
r_{\#}:\pilip{1}(E,e_0)\longrightarrow\pilip{1}(M,p_0)
\]
is injective.
\end{lemma}

\begin{proof}
Let homotopy class $[\alpha]\in\pilip{1}(E,e_0)$ be in the kernel of the homomorphism $r_{\#}$. Then, for any representative $\alpha$, the  loop $r\circ \alpha$ in $M$ is null-homotopic. So, there exists a Lipschitz map $H:\D^2\rightarrow M$ such that the following solid diagram commutes:
\begin{center}
\begin{tikzcd}
\S^1 \arrow[rr, "\alpha"] \arrow[dd, hookrightarrow] & & E \arrow[dd, "r"] \\ \\
\D^2 \arrow[rr, "H"]\arrow[rruu, dashed, "\r{H}", "\exists"'] & & M.
\end{tikzcd}
\end{center}

Since the based map $r:(E,e_0)\rightarrow(M,p_0)$ satisfies the unique lifting property and the 2-disk $\D^2$ is quasi-convex and Lipschitz simply connected, there exists a lift of the map $H$ to a Lipschitz map $\r{H}:\D^2\rightarrow E$ as indicated in the diagram. So, the loop $\alpha$ in the metric space $E$ is null-homotopic and thus $[\alpha]=[e_0]$. Therefore, the kernel of the homomorphism $r_\#$ is trivial.
\end{proof}

As a matter of note, Lemma~\ref{induced r is injective} does not rely on the uniqueness of a lift, but rather only on the existence of a lift. See Theorem~\ref{unique lifting} for how to construct the necessary lift in the instance that the map $\pi:(\ps{M}{d},[p_0])\rightarrow (M,p_0)$ is considered.

\begin{theorem}\label{upl implies simply connected}
If $\pi:(\ps{M}{d},[p_0])\rightarrow (M,p_0)$ satisfies the unique lifting property, then the universal path space $\ps{M}{d}$ is Lipschitz simply connected.
\end{theorem}

\begin{proof}
Assume $\pi:(\ps{M}{d},[p_0])\rightarrow (M,p_0)$ satisfies the unique lifting property. By Lemma~\ref{induced r is injective}, the induced homomorphism $\pi_\#$ is injective. Thus, it is enough to show that the image of the homomorphism $\pi_\#$ is trivial in the group $\pilip{1}(M,p_0)$ in order to show that $\pilip{1}(\ps{M}{d},[p_0])$ is trivial.

Let the homotopy class $[\alpha]$ of loops based at $p_0$ be in the image of the homomorphism $\pi_\#$. So, there exists a representative $\alpha:I\rightarrow M$ in $[\alpha]$ such that for some based loop $A:I\rightarrow\ps{M}{d}$, we have equality of maps $\alpha=\pi\circ A$. That is to say, the loop $A$ is a lift of the loop $\alpha$. By assumption, the map $\pi$ satisfies the unique lifting property. So, the lift of the loop $\alpha$  in $\ps{M}{d}$ must be $A=\r{\alpha}$, the lift of a path described in Section \ref{def of lift}. Since $\r{\alpha}$ is a loop, the initial and terminal points of the path $\r{\alpha}$ are equal, 
\[
[p_0]=\r{\alpha}(0)=\r{\alpha}(1)=[\alpha].
\]
Thus, $\alpha$ is null-homotopic and therefore the image of $\pi_\#$ is the trivial subgroup in $\pilip{1}(M,p_0)$.
\end{proof}

Theorem~\ref{upl implies simply connected} offers a converse to the lifting requirements of the unique path lifting property, which we state as the following result of independent interest.

\begin{corollary}
Suppose that endpoint projection $\pi:(\ps{M}{d},[p_0])\rightarrow (M,p_0)$ satisfies the unique lifting property. Let $X$ be a based quasi-convex metric space and let $f:X\rightarrow M$ be a based Lipschitz map. Then, there exists a unique lift $\r{f}:X\rightarrow\ps{M}{d}$ of the map $f$ if and only if the homomorphism induced by $f$ between first Lipschitz homotopy groups
\[
f_\#:\pilip{1}(X,x_0) \longrightarrow \pilip{1}(M,p_0)
\]
is trivial at $[p_0]$.
\end{corollary}

\begin{proof}
Since the map $\pi:(\ps{M}{d},[p_0])\rightarrow (M,p_0)$ satisfies the unique lifting property, if $f_\#$ is trivial, then there is a unique lift of $f$ (See Definition~\ref{unique path lifting definition}). 

Now assume that there is a unique lift of the map $f$. The map $f$ can be expressed as the composition $\pi\circ\r{f}$. Thus, the map induced on first Lipschitz homotopy groups can also be expressed as a composition:
\[
f_\#=(\pi\circ \r{f})_\#=\pi_\#\circ\r{f}_\#.
\]
Since, by Theorem~\ref{upl implies simply connected}, the universal path space $\ps{M}{d}$ is Lipschitz simply connected, the homomorphism $\r{f}_\#$ is trivial. Therefore, so is the homomorphism $f_\#$.
\end{proof}

\subsection{The universal property of the universal path space}


Before turning our attention to the universal path space over a $\D^2$-unrectifiable metric space, we establish a universal property for the universal path space. First, we return to the metric structure of the universal path space and argue that $\ps{M}{d}$ is a length space whenever it is a metric space. In the following proof, the path $A$ is a lift of the path $\alpha$, though this fact has no bearing on the argument.

\begin{theorem} \label{lem:length}
For a based metric space $M$ such that $\ps{M}{d}$ is a metric space, the universal path space $\ps{M}{d}$ is a length space. 
\end{theorem}

\begin{proof}
Let $[\gamma],[\eta] \in \ps{M}{d}$. Let $C > 1$ and choose a representative $\alpha \in [\overline{\gamma} * \eta]$ such that $\length{M}(\alpha) \leq C \cdot \dps([\gamma],[\eta])$. By choosing an arc length parametrization, we may assume that $\lip{\alpha} = \length{M}(\alpha)$. 

Consider the path in $A:I \to \ps{M}{d}$ given by $t \mapsto [\eta * \overline{\alpha}_{1-t}]$, where $\overline{\alpha}_{1-t}$ is the inverse path $\overline{\alpha}$ reparametrized  with respect to the value $1-t\in I$ via the method described above. The path $A$ starts at $[\gamma]$ and ends at $[\eta]$. As will now be verified, the path $A$ is in fact Lipschitz. To see this, let $s,t\in I$ where $s<t$ and consider
\[
\dps(A(s),A(t))=\inf\left\{\length{M}(\beta)~:~\beta\simeq \overline{(\eta*\overline{\alpha}_{1-s})}*(\eta*\overline{\alpha}_{1-t})\right\}
\]
The path $\alpha|_{[s,t]}$ is among the paths $\beta$ that the infimum is taken with respect to, as
\begin{eqnarray*}
\overline{(\eta*\overline{\alpha}_{1-s})}*(\eta*\overline{\alpha}_{1-t}) & \simeq & (\overline{\overline{\alpha}_{1-s}}*\overline{\eta})*(\eta*\overline{\alpha}_{1-t}) \\
& \simeq &  \overline{\overline{\alpha}_{1-s}}*\overline{\alpha}_{1-t} \\
& \simeq & \alpha|_{[s,t]}.
\end{eqnarray*}
 So,
\[
\dps(A(s),A(t))\leq\length{M}(\alpha|_{[s,t]})\leq\lip{\alpha}\cdot|s-t|.
\]
Thus, $A$ is Lipschitz and $\lip{A}\leq\lip{\alpha}$.

Therefore, the length of the path $A$ is bounded above by
\[
\length{\ps{M}{d}}(A)\leq \lip{A}\leq\lip{\alpha} \leq C \cdot \dps([\gamma],[\eta]).
\]
 Since $C > 1$ was arbitrary, the universal path space $\ps{M}{d}$ is a length space.
\end{proof}

\begin{theorem}[Universal Property]
Let $(M,p_0)$ be a based metric space such that the map $\pi:(\ps{M}{d},[p_0])\rightarrow (M,p_0)$ satisfies the unique lifting property and the universal path space $\ps{M}{d}$ is a metric space. Let $(E,e_0)$ be a based Lipschitz simply connected quasi-convex space. Suppose that $r:(E,e_0)\rightarrow (M,p_0)$ is a based Lipschitz map satisfying the unique lifting property.
Then the metric spaces $E$ and $\ps{M}{d}$ are biLipschitz equivalent with commutative diagram
\begin{center}
\begin{tikzcd}
(E,e_0) \arrow[dr, "r"']\arrow[rr, "\cong"] & & (\ps{M}{d},[p_0])\arrow[dl, "\pi"] \\ 
& (M,p_0). &
\end{tikzcd}
\end{center}
If in addition, $M$ is a geodesic space, $E$ is a length space, $\Lip(r)= 1$, and for any based Lipschitz map $f:(X,x_0)\rightarrow(M,p_0)$ with length space domain which induces the trivial homomorphism $f_{\#}$, the lift $\r{f}$ of the map $f$ with respect to $r:E\rightarrow M$ satisfies $\Lip(\r{f})=\Lip(f)$, then $E$ and $\ps{M}{d}$ are isometric.
\end{theorem}

\begin{proof}
First, some observations about the universal path space $\ps{M}{d}$. 
By Theorem~\ref{upl implies simply connected}, the universal path space is Lipschitz simply connected. Thus, the homomorphism $\pi_\#:\pilip{1}(\ps{M}{d},[p_0])\rightarrow\pilip{1}(M,p_0)$ is trivial. Also, by Theorem~\ref{lem:length}, the universal path space is a length space, and thus quasi-convex.

Since the map $r:(E,e_0)\rightarrow (M,p_0)$ satisfies the unique lifting property, there is a lift $\r{\pi}:(\ps{M}{d},[p_0])\rightarrow(E,e_0)$ of $\pi$ with respect to $r$ such that $r\circ\r{\pi}=\pi$. Likewise, since $\pi:(\ps{M}{d},[p_0])\rightarrow (M,p_0)$ satisfies the unique lifting property,
there is a lift $\r{r}:(E,e_0)\rightarrow(\ps{M}{d},[p_0])$ of $r$ with respect to $\pi$ such that $\pi\circ\r{r}=r$. Thus, by the unique lifting property, $\r{r}\circ\r{\pi}=\id_{\ps{M}{d}}$ and $\r{\pi}\circ\r{r}=\id_{E}$. So, the Lipschitz maps $\r{\pi}$ and $\r{r}$ are bijections and inverses of each other.

Additionally, the maps $\r{\pi}$ and $\r{r}$ are biLipschitz. Indeed, let $e_1, e_2\in E$ and denote by $d_E$ the metric on $E$. Then,
\begin{eqnarray*}
d_E(e_1, e_2)     & = & d_E(\r{\pi}\circ\r{r}(e_1),\r{\pi}\circ\r{r}(e_2)) \\
				& \leq & \Lip(\r{\pi})\cdot\dps(\r{r}(e_1),\r{r}(e_2)) \\
				& \leq & \Lip(\r{\pi})\cdot \Lip(\r{r})\cdot d_E(e_1, e_2).
\end{eqnarray*}
Thus, the map $\r{r}$ is biLipschitz as
\begin{equation}\label{biLipschitz r}
\frac{1}{\Lip(\r{\pi})}\cdot d_E(e_1, e_2)  \leq \dps(\r{r}(e_1),\r{r}(e_2)) \leq \Lip(\r{r})\cdot d_E(e_1, e_2).
\end{equation}
Similarly, the map $\r{\pi}$ is biLipschitz as, for $[\gamma_1],[\gamma_2]\in\ps{M}{d}$,
\begin{equation}\label{biLipschitz pi}
\frac{1}{\Lip(\r{r})}\cdot \dps([\gamma_1], [\gamma_2])  \leq d_E(\r{\pi}[\gamma_1],\r{\pi}[\gamma_2]) \leq \Lip(\r{\pi})\cdot \dps([\gamma_1], [\gamma_2]).
\end{equation}

Now, assume $M$ is a geodesic space, $E$ is a length space,  $\Lip(r)=1$, and $\Lip(\r{f})=\Lip(f)$ for any map $f:(X,x_0)\rightarrow(M,p_0)$ with length space domain satisfying the requirements of the unique lifting property with respect to $r:E\rightarrow M$. Since $M$ is a geodesic space, $\Lip(\pi)=1$ by Lemma~\ref{endpoint projection is a metric map}. Then, $\Lip(\r{\pi})=1$ since $\ps{M}{d}$ is a length space. Also, since $E$ is a length space, by Theorem~\ref{unique lifting}, $\Lip(\r{r})=1$. Thus, inequalities (\ref{biLipschitz r}) and (\ref{biLipschitz pi}) yield the following equalities for any $e_1, e_2\in E$ and $[\gamma_1], [\gamma_2]\in\ps{M}{d}$,
\[
\dps(\r{r}(e_1),\r{r}(e_2)) =d_E(e_1, e_2) \text{ and } d_E(\r{\pi}[\gamma_1],\r{\pi}[\gamma_2])=\dps([\gamma_1], [\gamma_2]).
\]
Therefore, $\r{r}$ and $\r{\pi}$ are isometries. 

\end{proof}

\section{The universal path space of tree-like homotopy spaces}\label{D2 section}



We now define a class of metric spaces whose homotopies are tree-like in that they exhibit properties similar to homotopies that factor through a metric tree. 

\begin{definition}\label{def TLH}
A metric space $M$ is a \textit{tree-like homotopy space} if, for any paths $\gamma:I\rightarrow M$ and $\beta:I\rightarrow M$ that are homotopic rel endpoints, there exists a Lipschitz map $H':I\times I\rightarrow M$ and a path $\beta':I\rightarrow M$ such that:
\begin{itemize}
\item the map $H'$ is a homotopy from $\gamma$ to $\beta'$,
\item $\Lip(H')=\Lip(\gamma)$, 
\item $\Ima(H')\subset\Ima(\gamma)$,
\item $\Lip(\beta')\leq \Lip(\gamma)$, 
\item $\Ima(\beta')\subset\Ima(\gamma)$, and
\item $\length{M}(\beta')\leq\length{M}(\beta)$.
\end{itemize}
Additionally, if $\beta$ is arc length parametrized and a length minimizing path in the homotopy class $[\gamma]$, then $\beta'=\beta$.
\end{definition}

As exhibited in Lemma 3.8 and Theorem 3.10 in \cite{perry2023existence}, purely 2-unrectifiable spaces are tree-like homotopy spaces. Thus, the first Heisenberg group $\H^1$ as well as any contact 3-manifold endowed with a sub-Riemannian structure is a tree-like homotopy space. See \cite{Amb} and \cite{perry2020lipschitz} for arguments for why these spaces  are respectively purely 2-unrectifiable. The results in \cite{perry2023existence} follow from homotpies factoring through a metric tree per a factorization result of Wenger and Young \cite[Theorem 5]{Weg}.

\subsection{Unique path lifting for the universal path space of a tree-like homotopy space}

Tree-like homotopy spaces have several properties that are sufficient for the associated universal path space to be a length space and satisfy the unique path lifting property, namely that each homotopy class has a length minimizing path and all points support only locally trivial representation, as will be shown. Versions of the next two results originally appeared in \cite{perry2023existence} for purely 2-unrectifiable spaces.

\begin{theorem}\label{core result}
Let $M$ be a tree-like homotopy space. For any homotopy class $[\gamma]$ in $M$, there exists a length minimizing path $\core{\gamma}\in[\gamma]$. 
Moreover, for any representative $\gamma\in[\gamma]$ in the class, $\Ima(\core{\gamma})\subset\Ima(\gamma)$.
\end{theorem}

\begin{proof}
Let $M$ be a tree-like homotopy space. Let $[\gamma]$ be a homotopy class of paths and define $\lmin\deq\inf\{\length{M}(\gamma)~|~\gamma\in[\gamma]\}$ to be the infimum of all lengths of paths in $[\gamma]$. 
 
For each natural number $n$, let $\gamma_n\in[\gamma]$ be a path such that $\length{M}(\gamma_n)\leq\lmin+\frac{1}{n}$. Furthermore, since $\gamma_1$ is homotopic rel endpoints to $\gamma_n$, by Definition~\ref{def TLH}, we can assume that $\Lip(\gamma_n)\leq\Lip(\gamma_1)$ and $\Ima(\gamma_n)\subset\Ima(\gamma_1)$. Additionally, there is a homotopy $H_n:I\times I\rightarrow M$ from $\gamma_1$ to $\gamma_n$ such that $\Lip(H_n)=\Lip(\gamma_1)$ and $\Ima(H_n)\subset\Ima(\gamma_1)$.

Now, for any $n\in\NN$, $\Lip(\gamma_n)\leq\Lip(\gamma_1)$. Since the images of the paths in the sequence $(\gamma_n)$ are subsets of the compact set $\Ima(\gamma_1)$, by Arzel\`{a}-Ascoli theorem, there exists a subsequence $(\gamma_{n_k})$ that uniformly converges to a Lipschitz path $\core{\gamma}$. By lower semi-continuity of the length measure, $\length{M}(\core{\gamma})\leq\liminf_k\length{M}(\gamma_{n_k})$. In fact, due to how the sequence $(\gamma_n)$ was selected, $\length{M}(\core{\gamma})\leq\lmin$. 

We now want to show that $\core{\gamma}\in[\gamma]$. Associated to the subsequence $(\gamma_{n_k})$, there is a sequence of homotopies $(H_{n_k})$ such that $\Lip(H_{n_k})=\Lip(\gamma_1)$ for each homotopy in the sequence. 
Since $\Ima(H_{n_k})\subset\Ima(\gamma_1)$ for each $n_k$ and $\Ima(\gamma_1)$ is compact, by Arzel\`{a}-Ascoli theorem, there exists a subsequence $(H_{n_{k_j}})$ that converges uniformly to a Lipschitz map $\core{H}:I\times I\rightarrow M$. 

Now, $\core{H}|_{I\times\{0\}}=\gamma_1$ since $H_{n_{k_j}}|_{I\times\{0\}}=\gamma_1$ for all $n_{k_j}$. Also, since the paths $H_{n_{k_j}}|_{I\times\{1\}}=\gamma_{n_{k_j}}$ converge uniformly to $\core{\gamma}$, then $\core{H}|_{I\times\{1\}}=\core{\gamma}$. So, the map $\core{H}$ is a homotopy from $\gamma_1$ to $\core{\gamma}$. Therefore, $\core{\gamma}\in[\gamma]$ and thus $\length{M}(\core{\gamma})=\lmin$.

Now, let $\gamma\in[\gamma]$. Assume that the length minimizing path $\core{\gamma}$ is arc length parametrized. Since $\gamma\simeq\core{\gamma}$, by Definition~\ref{def TLH}, there is a homotopy from $\gamma$ to $\core{\gamma}$ such that $\Ima(\core{\gamma})\subset\Ima(\gamma)$. 
\end{proof}

A consequence of Theorem~\ref{core result} is that a length minimzing path for a homotopy class is unique up to reparametrization.

%
%
%
%
%

\begin{corollary}\label{p2u has metric on path space}
For a tree-like homotopy space $M$, every point $p\in M$ supports only trivial local representation. Thus, if $M$ is a based space, the lifted pseudo-metric $\dps$ is a metric and the universal path space $\ps{M}{d}$ is a length space.
\end{corollary}

\begin{proof}
Let $p\in M$ and $[\alpha]\in\pilip{1}(M,p)$ be a locally representable homotopy class of loops based at $p$. Then, by Theorem~\ref{core result}, $[\alpha]$ has a length minimizer $\core{\alpha}$ whose image is a subset of every neighborhood of $p$. Therefore, $\core{\alpha}$ is the constant loop at $p$ and thus $[\alpha]=[p]$.

Now assume for the purpose of defining the universal path space that $M$ is based. Since every point in $M$ supports only locally trivial representation, by Lemma~\ref{universal path space is a metric space}, the lifted pseudo-metric $\dps$ is a metric. By Theorem~\ref{lem:length}, the universal path space $\ps{M}{d}$ is a length space.
\end{proof}



{

\begin{theorem}\label{unique path lifting}
For any based tree-like homotopy space $(M,p_0)$ with metric $d$, the universal path space $\ps{M}{d}$ has the unique path lifting property.
\end{theorem}


\begin{proof}
Let $\gamma:(I,0)\rightarrow (M,p_0)$ be a based path. The path $\r{\gamma}$ is a lift of $\gamma$ (Lemma~\ref{the lift is Lipschitz}). Let $\Gamma:(I,0)\rightarrow(\ps{M}{d},[p_0])$ be another lift of the path $\gamma$ with Lipschitz constant $L=\Lip(\Gamma)$. We aim to show that $\Gamma=\r{\gamma}$. 

It is enough to show, for arbitrary based path $\gamma:(I,0)\rightarrow(M,p_0)$ and lift $\Gamma$, that $\gamma$ is a representative for the homotopy class $\Gamma(1)$, that is, $\Gamma(1)=[\gamma]=\r{\gamma}(1)$. For any $t\in I$, the equality $\Gamma(t)=\r{\gamma}(t)$ is a consequence of the case when $t=1$. Indeed, consider the path $\Gamma_t:I\rightarrow\ps{M}{d}$ defined by $\Gamma_t(s)=\Gamma(st)$.  Then, the map $\Gamma_t$ is based, Lipschitz, and for any $s\in I$,
\[
\pi\circ\Gamma_t(s)=\pi\circ\Gamma(st)=\gamma(st)=\gamma_t(s).
\]
So, $\Gamma_t$ is a lift of $\gamma_t$. Thus,
\[
\Gamma(t)=\Gamma_t(1)=[\gamma_t]=\r{\gamma}(t).
\]

We begin by constructing, for each non-negative integer $n\in\mathbb{N}_0$, a representative $\beta_n$ of $\Gamma(1)$ that is never more than $\frac{L}{2^{n-1}}$ away from the path $\gamma$. The sequence of paths $(\beta_n)$ can be thought of as successively better approximations of the path $\gamma$.

Let $n\in\mathbb{N}_0$. Partition the interval $I$ into $2^n$ subintervals, each of width $\frac{1}{2^n}$. For $j\in\{1,2,\ldots,2^n\}$, let $I_{n}^{j}\deq\left[\frac{j-1}{2^n},\frac{j}{2^n} \right]$ be the $j$th subinterval of the $n$th partition.

We proceed by induction on $j\in\{1,2,\ldots, 2^n\}$ to show that there exists a representative $\beta_n^{1,j}\in\Gamma\left(\frac{j}{2^n}\right)$ with domain $\left[0,\frac{j}{2^n}\right]$ and, for $i\in\{1,\ldots,j\}$, paths $\beta_n^i:I_{n}^{i}\rightarrow M$ with initial point $\gamma\left(\frac{i-1}{2^n}\right)$ and endpoint $\gamma\left(\frac{i}{2^n}\right)$ such that
\begin{itemize}
\item  $\ds\beta_n^{1,j}\left|_{I_{n}^{i}}\right.=\beta_n^i$, 
\item $\ds \length{M}\left(\beta_n^{i}\right)<\frac{L}{2^{n-1}}$, and
\item $\ds\Ima\left(\beta_n^i\right)=\Ima\left(\beta_n^{1,j}\left|_{I_{n}^{i}}\right.\right)\subset B\left(\gamma\left(\frac{j-1}{2^n}\right), \frac{L}{2^{n-1}} \right)$.
\end{itemize}

\begin{figure}
\begin{center}
\begin{tikzpicture}
]
\begin{axis}[
 axis line style={draw=none},
tick style={draw=none},
  ymin=-.75, ymax=1.4, ytick=\empty, 
  xmin=-5, xmax=7.5, xtick=\empty, 
  domain=-1:6,samples=101, 
]

\addplot[domain=-4:7.5, blue, thick]{(-0.01)*(x)*(x-6)} node[below left]{$\gamma$};
\addplot[domain=-4.02:-3, thick, name path=E]{(-0.25)*(x+4)*(x+0.6)-0.4} node[right]{\scriptsize $\beta_{n}^{1}$};
\addplot[domain=-3:-1, thick, name path=F]{(-0.25)*(x+4)*(x+0.6)-0.4};
\addplot[domain=-1:-0.5, thick, name path=F]{(-0.25)*(x-2)*(x+0.6)+0.2};
\addplot[domain=2.3:3, thick, name path=C]{(-0.2)*(x)*(x-3.1)};
\addplot[domain=3.02:3.8, thick, name path=E]{(-0.2)*(x-6)*(x-2.9)} node[right]{\scriptsize $\beta_{n}^{j}$};
\addplot[domain=3.8:6, thick, name path=F]{(-0.2)*(x-6)*(x-2.9)};



\draw[dashed] (70,40) circle [radius =2.3cm];

\draw[dashed] (800,85) circle [radius =1.8cm];


\addplot[soldot] coordinates{(-4,-0.4)} node[below] {\scriptsize $\displaystyle\gamma\left(0\right)=p_0$};
\addplot[soldot] coordinates{(-1,-0.07)} node[below] {\scriptsize $\displaystyle\gamma\left(\frac{1}{2^{n}}\right)$};
\addplot[mark=none] coordinates{(1,-0.07)} node[below] {\scriptsize $\ldots$};
\addplot[mark=none] coordinates{(1,0.4)} node[below] {\scriptsize $\ldots$};
\addplot[soldot] coordinates{(3,0.09)} node[below] {\scriptsize $\displaystyle\gamma\left(\frac{j-1}{2^{n}}\right)$};
\addplot[soldot] coordinates{(6,0)} node[below left] {\scriptsize $\displaystyle\gamma\left(\frac{j}{2^n}\right)$};

\addplot[mark=none] coordinates{(-3,0.75)} node[below] {\scriptsize $B\left(\gamma(0),\frac{L}{2^{n-1}} \right)$};

\addplot[mark=none] coordinates{(5.5,1)} node[below] {\scriptsize $B\left(\gamma\left(\frac{j-1}{2^n}\right),\frac{L}{2^{n-1}} \right)$};

\end{axis}
\end{tikzpicture}
\caption{Image of $\beta_n^{1,j}$.}
\end{center}
\end{figure}
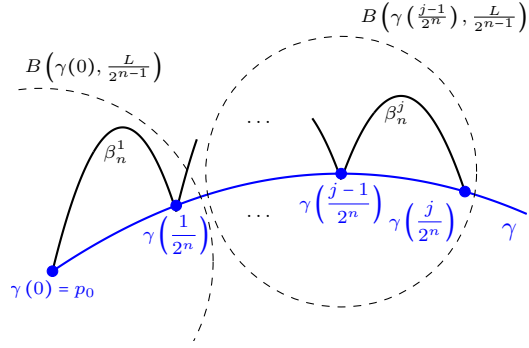



Let $j=1$ and consider the interval $I_{n}^{1}=\left[0,\frac{1}{2^n}\right]$. Since $\Gamma$ is $L$-Lipschitz, we have the following inequality:
\[
\dps\left(\Gamma(0),\Gamma\left( \frac{1}{2^n}\right) \right)\leq L\cdot\left|0-\frac{1}{2^n}\right|=\frac{L}{2^n}.
\]
Recalling the definition of the lifted metric $\dps$ and noting that $\Gamma(0)=[p_0]$, we have the inequality
\[
\inf\left\{~\length{M}(\beta)~:~\beta\in\Gamma\left( \frac{1}{2^n}\right)~\right\}=\dps\left([p_0],\Gamma\left( \frac{1}{2^n}\right) \right)=\dps\left(\Gamma(0),\Gamma\left( \frac{1}{2^n}\right)\right)\leq\frac{L}{2^n}.
\]
Thus, there exists a representative $\beta_n^{1}\in\Gamma\left( \frac{1}{2^n}\right)$ with bounded length, 
\[
 \length{M}\left(\beta_n^1\right)<\frac{L}{2^n}+\frac{L}{2^n}=\frac{L}{2^{n-1}}. 
\]
The representative $\beta_n^{1}:I_{n}^{1}\rightarrow M$ is a path with initial point $\gamma(0)=p_0$ and with endpoint $\gamma\left(\frac{1}{2^n} \right)$.
Since the value $\frac{L}{2^{n-1}}$ is an upper bound for the length of $\beta_n^1$, the image of $\beta_n^{1}$ is a subset of the open ball centered at $\gamma(0)$ of radius $\frac{L}{2^{n-1}}$:
\[
\Ima\left(\beta_n^{1}\right)\subset B\left(\gamma(0),\frac{L}{2^{n-1}} \right).
\]
Set $\beta_n^{1,1}\deq\beta_n^{1}$.

Now, for some $j\in\{1,2,\ldots, 2^{n}-1\}$, suppose  a representative $\beta_n^{1,j}\in\Gamma\left( \frac{j}{2^n}\right)$ exists with domain $\left[0,\frac{j}{2^n} \right]$ 
such that, for $i\in\{1,\ldots,j\}$, the image of the restricted path 
\[
\beta_n^{i}=\beta_n^{1,j}\left|_{I_{n}^{i}}\right.
\]
 is a subset of the open ball $B\left(\gamma\left(\frac{i-1}{2^n}\right), \frac{L}{2^{n-1}} \right)$ and the path $\beta_n^i$ has length less than $\frac{L}{2^{n-1}}$. The endpoint of the path $\beta_n^{1,j}$ is the point $\beta_n^j\left(\frac{j}{2^n}\right)=\gamma\left(\frac{j}{2^n}\right)$.

Consider the subinterval $I_{n}^{j+1}=\left[\frac{j}{2^n},\frac{j+1}{2^n} \right]$. Again, since $\Gamma$ is $L$-Lipschitz, we have the inequality:
\[
\dps\left(\Gamma\left(\frac{j}{2^n}\right),\Gamma\left( \frac{j+1}{2^n}\right) \right)\leq L\cdot\left| \frac{j}{2^n}-\frac{j+1}{2^n}\right|=\frac{L}{2^n}.
\]
Via the definition of the lifted metric $\dps$, we have the inequality
\[
\inf\left\{~\length{M}(\beta)~:~\beta\simeq\overline{\eta_1}*\eta_2\text{ where }  \eta_1\in\Gamma\left(\frac{j}{2^n}\right)\text{ and }\eta_2\in\Gamma\left( \frac{j+1}{2^n}\right)~\right\} \leq \frac{L}{2^n}
\]
Thus, there exists a path $\beta_n^{j+1}:I_{n}^{j+1}\rightarrow M$ with initial point $\gamma\left(\frac{j}{2^n}\right)=\pi\circ\Gamma\left(\frac{j}{2^n}\right)$ and endpoint $\gamma\left(\frac{j+1}{2^n}\right)=\pi\circ\Gamma\left(\frac{j+1}{2^n}\right)$ that has length
\[
 \length{M}\left(\beta_n^{j+1}\right) <\frac{L}{2^n}+\frac{L}{2^n}=\frac{L}{2^{n-1}}.
\]
 So, the image of $\beta_n^{j+1}$ is a subset of the open ball centered at $\gamma\left(\frac{j}{2^n}\right)$ of radius $\frac{L}{2^{n-1}}$:
\[
\Ima(\beta_n^{j+1})\subset B\left(\gamma\left(\frac{j}{2^n}\right),\frac{L}{2^{n-1}} \right).
\]
Moreover, the path $\beta_n^{j+1}$ is homotopic to the reverse of any representative in $\Gamma\left(\frac{j}{2^n}\right)$ concatenated with any representative of $\Gamma\left( \frac{j+1}{2^n}\right)$.  Since the path $\beta_n^{1,j}$ of the inductive hypothesis is a representative of $\Gamma\left(\frac{j}{2^n}\right)$, the Lipschitz path $\beta_n^{1,j+1}:\left[0,\frac{j+1}{2^n}\right]\rightarrow M$ piecewise defined by
\[ \beta_n^{1,j+1}(t)\deq\begin{cases} 
      \beta_n^{1,j}(t) & \text{ if }t\in\left[0,\frac{j}{2^n}\right] \\ \\
     \beta_n^{j+1}(t) & \text{ if }t\in I_{n}^{j+1}      
   \end{cases}
\]
is a representative of $\Gamma\left( \frac{j+1}{2^n}\right)$.

Thus, for each $n\in\NN_0$, there exists a representative $\beta_n\deq\beta_n^{1,{2^n}}\in\Gamma(1)$ 
such that for $j\in\{1,2,\ldots,2^n\}$,
\begin{equation}\label{restricted to elements of Gamma}
\beta_n\left|_{\left[0,\frac{j}{2^{n}}\right]}\right.=\beta_n^{1,j}\in\Gamma\left(\frac{j}{2^n}\right).
\end{equation}
Also, the path $\beta_n$ when restricted to the subinterval $I_{n}^{j}$ is equal to the path 
\begin{equation}\label{restrictions}
\beta_n\left|_{I_{n}^{j}}\right.=\beta_n^j
\end{equation}
where the path $\beta_n^j:I_{n}^{j}\rightarrow M$ has initial point $\beta_n^j\left(\frac{j-1}{2^n}\right)=\gamma\left(\frac{j-1}{2^n}\right)$, endpoint $\beta_n^j\left(\frac{j}{2^n}\right)=\gamma\left(\frac{j}{2^n}\right)$, and bounded length
\begin{equation}\label{length bound on bits}
 \length{M}\left(\beta_n^{j}\right)<\frac{L}{2^{n-1}}.
\end{equation}
Additionally,  the image of the restriction $\beta_n\left|_{I_{n}^{j}}\right.$ is within $\frac{L}{2^{n-1}}$ of the initial point, 
\begin{equation}\label{open ball for N,k}
\Ima\left(\beta_n^j\right)=\Ima\left(\beta_n\left|_{I_{n}^{j}}\right.\right)\subset B\left(\gamma\left(\frac{j-1}{2^n}\right), \frac{L}{2^{n-1}} \right).
\end{equation}


Moreover, for each $n\in\NN_0$ and each $j\in\{1,2,\ldots,2^n\}$, via Theorem~\ref{core result}, take the path $\beta_n^j$ to be the arc length parametrized length minimizing path of the homotopy class  $[\beta_n^j]$. Thus, by construction, the path $\beta_0$ is the length minimizing path in the class $\Gamma(1)$.

As is now shown, the sequence of paths $(\beta_n)$ has a uniform bound on the associated Lipschitz constants.

\begin{lemma}\label{uniform bound on lip for eta_N}
For all $n\in\NN_0$, $\Lip(\beta_n)< 2L$ 
\end{lemma}

\begin{proof}
Let $n\in\NN_0$. For $j\in\{1,2,\ldots,2^n\}$, since the path $\beta_n^j$ is arc length parametrized and has length less than $\frac{L}{2^{n-1}}$ by (\ref{length bound on bits}),
\[
\Lip(\beta_n^j)=\frac{\length{M}(\beta_n^j)}{\frac{j}{2^{n}}-\frac{j-1}{2^{n}}}=2^n\length{M}(\beta_n^j)<2^n\cdot\frac{L}{2^{n-1}}=2L.
\]

Now, $\beta_n\left|_{I_{n}^{j}}\right.=\beta_n^j$ for each $j$ by (\ref{restrictions}). Thus, the Lipschitz constant of the path $\beta_n$ is bounded above by $2L$ as well. 
\end{proof}

As will be argued, the sequence of paths $(\beta_n)$ converges pointwise to the path $\gamma$. First note that, by Lemma~\ref{lift lip bound}, the Lipschitz constant $L_\gamma$ associated to the path $\gamma$ is no larger than the Lipschitz constant associated with the lift $\Gamma$, that is, $L_\gamma\leq L$. 



\begin{lemma}\label{pointwise convergence}
$\beta_n\longrightarrow\gamma$ as $n\rightarrow\infty$.
\end{lemma}
\begin{proof} 
Let $t\in I$. For any $n\in\NN_0$, there exists $j\in\{1,2,\ldots,2^n\}$ such that $t\in I_{n}^{j}=\left[\frac{j-1}{2^n},\frac{j}{2^n}\right]$. By (\ref{open ball for N,k}), $\beta_n(t)\in B\left(\gamma\left(\frac{j-1}{2^n}\right),\frac{L}{2^{n-1}} \right)$.  Thus,
\begin{eqnarray*}
d(\beta_n(t),\gamma(t)) & \leq & d\left(\beta_n(t),\gamma\left(\frac{j-1}{2^n}\right)\right)+d\left(\gamma\left(\frac{j-1}{2^n}\right),\gamma(t)\right) \\
					   & \leq & \left(\frac{L}{2^{n-1}}\right) + L_\gamma\left|\frac{j-1}{2^n}-t\right| \\
					   & \leq & \left(\frac{L}{2^{n-1}}\right)  + L_\gamma\left(\frac{1}{2^n}\right) \\
					   & \leq & \left(\frac{L}{2^{n-1}}\right)  + L\left(\frac{1}{2^n}\right) \\
					   & = & \frac{3L}{2^{n}}.
\end{eqnarray*}
Therefore, $\ds\lim_{n\rightarrow \infty}\beta_n(t)=\gamma(t)$.
\end{proof}

Now, let $n\in\NN_0$. We will further investigate the relationship between the paths $\beta_n$, $\beta_{n+1}$, and $\gamma$, eventually concluding that the image of $\beta_n$ is a subset of the image of $\gamma$.

By (\ref{restricted to elements of Gamma}), for any $j\in\{1,2,\ldots,2^n\}$, the restricted paths $\beta_n\left|_{\left[0,\frac{j}{2^n}\right]}\right.$ and $\beta_{n+1}\left|_{\left[0,\frac{j}{2^n}\right]}\right.$ are representatives in $\Gamma\left(\frac{j}{2^{n}}\right)$. So, for any $j$,
\[
\beta_{n}\left|_{\left[0,\frac{j-1}{2^n}\right]}\right.\simeq\beta_{n+1}\left|_{\left[0,\frac{j-1}{2^n}\right]}\right. 
\text{ and }
\beta_{n}\left|_{\left[0,\frac{j}{2^n}\right]}\right.\simeq\beta_{n+1}\left|_{\left[0,\frac{j}{2^n}\right]}\right. .
\]
Thus, we have the following homotopy equivalences,
\begin{eqnarray*}
\beta_{n+1}\left|_{I_{n}^{j}}\right. & \simeq & \overline{\beta_{n+1}\left|_{\left[0,\frac{j-1}{2^n}\right]}\right.}*\beta_{n+1}\left|_{\left[0,\frac{j}{2^n} \right]}\right.  \\
 & \simeq & \overline{\beta_n\left|_{\left[0,\frac{j-1}{2^n}\right]}\right.}*\beta_{n}\left|_{\left[0,\frac{j}{2^n}\right]}\right. \\
  & \simeq & \beta_n\left|_{I_{n}^{j}}\right. \\
  & = & \beta_n^j.
\end{eqnarray*}

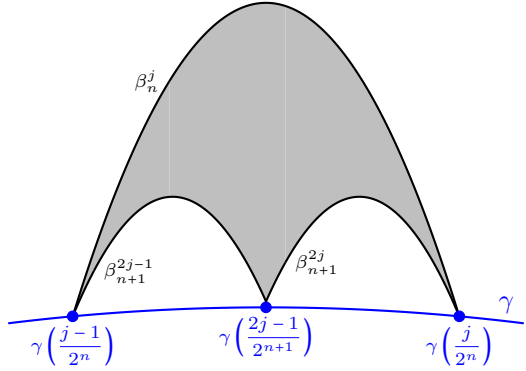
\begin{figure}\label{relationship}
\begin{center}
\begin{tikzpicture}
]
\begin{axis}[
 axis line style={draw=none},
tick style={draw=none},
  ymin=-1, ymax=3.3, ytick=\empty, 
  xmin=-1, xmax=7, xtick=\empty, 
  domain=-1:6,samples=101, 
]

\addplot[domain=-1:7, blue, thick]{(-0.01)*(x)*(x-6)} node[above left]{$\gamma$};
\addplot[domain=0:0.33, thick, name path=A]{(-0.5)*(x)*(x-3.1)} node[right]{\scriptsize $\beta_{n+1}^{2j-1}$};
\addplot[domain=0.33:3.02, thick, name path=C]{(-0.5)*(x)*(x-3.1)};
\addplot[domain=2.98:3.3, thick, name path=E]{(-0.5)*(x-6)*(x-2.9)} node[right]{\scriptsize $\beta_{n+1}^{2j}$};
\addplot[domain=3.3:6, thick, name path=F]{(-0.5)*(x-6)*(x-2.9)};
\addplot[domain=0:1.5, thick, name path=B]{(-0.35)*(x)*(x-6)} node[left]{\scriptsize $\beta_n^j$};
\addplot[domain=1.5:6, thick, name path=D]{(-0.35)*(x)*(x-6)};

\addplot[gray!50] fill between[of=A and B,soft clip={domain=0:0.33}]; 
\addplot[gray!50] fill between[of=C and B,soft clip={domain=0.33:1.5}]; 
\addplot[gray!50] fill between[of=C and D,soft clip={domain=1.5:3}];
\addplot[gray!50] fill between[of=E and D,soft clip={domain=3:3.3}];
\addplot[gray!50] fill between[of=F and D,soft clip={domain=3.3:6}];


\addplot[soldot] coordinates{(0,0)} node[below] {\scriptsize $\displaystyle\gamma\left(\frac{j-1}{2^n}\right)$};
\addplot[soldot] coordinates{(3,0.09)} node[below] {\scriptsize $\displaystyle\gamma\left(\frac{2j-1}{2^{n+1}}\right)$};
\addplot[soldot] coordinates{(6,0)} node[below] {\scriptsize $\displaystyle\gamma\left(\frac{j}{2^n}\right)$};



\end{axis}
\end{tikzpicture}
\caption{$\beta_n^j\simeq \beta_{n+1}\left|_{I_{n}^{j}}\right.$, which is comprised of $\beta_{n+1}^{2j-1}$ and $\beta_{n+1}^{2j}$.}
\end{center}
\end{figure}

So, $\beta_{n+1}\left|_{I_{n}^{j}}\right.\in[\beta_n^j]$. Moreover, since the path $\beta_n^j$ is the length minimizer of the class $[\beta_n^j]$, by Theorem~\ref{core result},
\[
\Ima(\beta_n^j)\subset\Ima\left(\beta_{n+1}\left|_{I_{n}^{j}}\right.\right).
\]
Note $I_{n}^{j}=I_{n+1}^{2j-1}\cup I_{n+1}^{2j}$. Also, by (\ref{restrictions}),
\[
\left(\beta_{n+1}\left|_{I_{n}^{j}}\right.\right)\left|_{I_{n+1}^{2j-1}}\right.=\beta_{n+1}\left|_{I_{n+1}^{2j-1}}\right.=\beta_{n+1}^{2j-1}, 
\]
and
\[
\left(\beta_{n+1}\left|_{I_{n}^{j}}\right.\right)\left|_{I_{n+1}^{2j}}\right.=\beta_{n+1}\left|_{I_{n+1}^{2j}}\right.=\beta_{n+1}^{2j}.
\]
Thus, we have the following relationship between the images of the paths $\beta_n^j$, $\beta_{n+1}^{2j-1}$, and $\beta_{n+1}^{2j}$:
\begin{equation*}\label{subset sequence of cores}
\Ima\left(\beta_n^j\right)\subset\Ima\left(\beta_{n+1}\left|_{I_{n}^{j}}\right.\right)=\Ima\left(\beta_{n+1}^{2j-1}\right)\cup\Ima\left(\beta_{n+1}^{2j}\right).
\end{equation*}
Therefore, by induction, for any $N\geq n$,
\begin{equation}\label{eta N in all eta n}
\Ima\left(\beta_n\right)\subset\bigcup_{j=1}^{2^N}\Ima\left(\beta_N^j\right).
\end{equation}
%
So, by (\ref{open ball for N,k}) and (\ref{eta N in all eta n}) , for all $N\geq n$, the path $\beta_n$ maps into a cover of $\Ima(\gamma)$ by open balls of radius $\frac{L}{2^{N-1}}$:
\[
\Ima\left(\beta_n\right)\subset\bigcup_{j=1}^{2^N}\Ima\left(\beta_N^j\right)\subset \bigcup_{j=1}^{2^N}B\left(\gamma\left(\frac{j-1}{2^N}\right),\frac{L}{2^{N-1}}\right).
\]
As the radii of these open balls limits to $0$ as $N$ goes to infinity, we gather the following result: 
\begin{lemma}\label{each eta maps into gamma}
For each $n\in\NN_0$, $\ds\Ima(\beta_n)\subset \Ima(\gamma)$.
\end{lemma}

Now, for $n\in\NN_0$, consider the path $\beta_n$. Since the paths $\beta_n$ and $\beta_0$ are both representatives of the class $\Gamma(1)$, there is a homotopy $H_n:I\times I\rightarrow M$ from $\beta_n$ to $\beta_0$. Moreover, since $M$ is a tree-like homotopy space, the homotopy $H_n$ can be taken such that $\Lip(H_n)=\Lip(\beta_n)$ and $\Ima(H_n)\subset\Ima(\beta_n)$. So, by Lemma~\ref{uniform bound on lip for eta_N}, $\Lip(H_n)<2L$. Also, by Lemma~\ref{each eta maps into gamma}, $\Ima(H_n)\subset\Ima(\gamma)$. Note that the homotopy $H_n$ guaranteed by Definition~\ref{def TLH} is from $\beta_n$ to $\beta_0$ since the path $\beta_0$ is arc length parametrized and the length minimizing path in the class $\Gamma(1)$.

So, there is a sequence of homotopies $(H_{n})$ that have a uniform bound of $2L$ for the associated Lipschitz constants. Since for all $n\in\NN_0$, the image of $H_n$ is a subset of the compact set $\Ima(\gamma)$, by Arzel\`{a}-Ascoli theorem, there exists a subsequence $(H_{n_{k}})$ such that the maps $H_{n_{k}}$ converge uniformly to a Lipschitz map $H_\infty:I\times I\rightarrow M$.

Now, $H_\infty|_{I\times\{1\}}=\beta_0$ since $H_{n_{k}}|_{I\times\{1\}}=\beta_0$ for all $n_{k}$. Also, since $H_{n_{k}}|_{I\times\{0\}}=\beta_{n_{k}}$ converges pointwise to $\gamma$ (Lemma~\ref{pointwise convergence}), then $H_\infty|_{I\times\{0\}}=\gamma$. Thus, the map $H_\infty$ is a  homotopy from $\gamma$ to $\beta_0$. Therefore, $\gamma\in\Gamma(1)$.

\end{proof}

%
%

%
%
%
%

%

We finish  with a proof of the primary result of this paper. Since purely 2-unrectifiable spaces are tree-like homotopy spaces, Theorem~\ref{main} follows immediately from Corollary~\ref{unique path lifting yields simply connected}.

\begin{corollary}\label{unique path lifting yields simply connected}
Let $M$ be a based tree-like homotopy space. Then endpoint projection $\pi:\ps{M}{d}\rightarrow M$ satisfies the unique lifting property and the universal path space $\ps{M}{d}$ is a Lipschitz simply connected length space.
\end{corollary}


\begin{proof}
%

By Corollary~\ref{p2u has metric on path space}, the universal path space $\ps{M}{d}$ is a length space. By Theorem~\ref{unique path lifting}, for a tree-like homotopy space $M$, the Lipschitz map $\pi:\ps{M}{d}\rightarrow M$ satisfies the unique path lifting property. Thus, by Theorem~\ref{unique lifting}, the map $\pi$ satisfies the unique lifting property. Therefore, by Theorem~\ref{upl implies simply connected}, the universal path space is Lipschitz simply connected.
\end{proof}

\bibliography{bib}{}
\bibliographystyle{plain}

\end{document}